\documentclass[11pt]{article}
\usepackage[T1]{fontenc}
\usepackage{arxiv}
\usepackage{tikz}

\usepackage{graphicx}
\usepackage[numbers,square]{natbib}
\usepackage{hyperref}
\usepackage{bookmark}
\usepackage{url}
\usepackage{booktabs}
\usepackage{amsfonts}
\usepackage{nicefrac}
\usepackage{amsmath,amssymb,amsthm}
\usepackage{algorithm,algpseudocode}
\usepackage{wrapfig}
\usepackage{enumerate}
\usepackage{bm}
\usepackage{xcolor}

\usepackage{pgfplots}
\pgfplotsset{compat=newest}
\usepackage[final]{showlabels}

\theoremstyle{plain}
\newtheorem{theorem}{Theorem}
\newtheorem{lemma}[theorem]{Lemma}
\newtheorem{proposition}[theorem]{Proposition}

\theoremstyle{definition}
\newtheorem{definition}[theorem]{Definition}

\theoremstyle{remark}

\numberwithin{equation}{section}
\numberwithin{figure}{section}
\numberwithin{table}{section}

\newcommand{\argmax}{\operatornamewithlimits{argmax}}

\newcommand{\EVal}{\mathbb{E}}
\newcommand{\E}[2]{\EVal_{#1}\left[  {#2}\right]}
\newcommand{\VVal}{\mathbb{V}}
\newcommand{\V}[2]{\VVal_{#1}\left[#2\right]}

\newcommand{\forward}{\mathbb{P}_{\mathrm{F}}}

\newcommand{\R}{\mathbb{R}}

\newcommand{\transpose}{\top}
\newcommand{\dd}{\mathrm{d}}
\newcommand{\KL}{\mathcal{D}_{\mathrm{KL}}}

\newcommand{\SampleSp}{\mathcal{X}}
\newcommand{\SearchSp}{\Theta}
\newcommand{\ContextSp}{\mathcal{C}}
\newcommand{\sampleDist}[1]{p_{#1}}
\newcommand{\searchDist}[1]{q_{#1}}
\newcommand{\baseObjFunc}{F}
\newcommand{\smoothedObjFunc}{\mathcal{F}}

\newcommand{\normal}[3]{N\left({#1} \mid {#2}, {#3}\right)}
\newcommand{\identityMatrix}{\mathbf{I}}
\newcommand{\transitMatrix}{\mathbf{T}}
\newcommand{\genotypeEstimator}{\overline{\frac{\partial}{\partial \psi}}^{\mathrm{g}}}

\newcommand{\popSize}{N}
\newcommand{\suptime}[2]{{#1}^{(#2)}}
\newcommand{\learningRate}{\eta}

\newcommand{\fullGradientEstimator}{\overline{\frac{\partial}{\partial \psi}}^{\mathrm{p}}}
\newcommand{\fullGradientGeneralArg}[1]{\overline{\frac{\partial}{\partial {#1}}}^{\mathrm{p}}}
\newcommand{\FisherInfo}[1]{\mathcal{I}\left[{#1}\right]}
\newcommand{\estimatorChar}{U}
\newcommand{\estimator}[1]{\estimatorChar\left({#1}\right)}
\newcommand{\RaoBlackwellize}[2]{\mathrm{RB}\left[{#1}\right]\left({#2}\right)}
\newcommand{\RaoBlackwellizeNoArgs}[1]{\mathrm{RB}\left[{#1}\right]}

\newcommand{\sufficientStatistic}[1]{u(#1)}
\newcommand{\baseMeasureExpFamilyChar}{h}

\newcommand{\actionSp}{\mathcal{A}}
\newcommand{\stateSp}{\mathcal{S}}

\newcommand{\diGamma}{\kappa}

\newcommand{\noise}{\xi}
\newcommand{\noiseDist}{q}
\newcommand{\thetaFunc}[2]{\beta_{#1}(#2)}
\newcommand{\reparameterizedEstimator}{\overline{\frac{\partial}{\partial \psi}}^{\mathrm{r}}}

\hypersetup{hypertexnames=false}

\bibpunct{[}{]}{;}{n}{,}{,}

\title{Accelerating Evolutionary Strategy via Rao-Blackwellizing Realization of Uncertain Input}

\renewcommand{\shorttitle}{Rao-Blackwellized ES under Input Uncertainty}
\renewcommand{\headeright}{A Preprint}
\renewcommand{\undertitle}{A Preprint}

\author{
  So Nakashima\thanks{\texttt{naaso0510@gmail.com}} \\
  Institute of Industrial Science \\
  The University of Tokyo \\
  \And
  Tetsuya J. Kobayashi\thanks{\texttt{tetsuya@mail.crmind.net}} \\
  Institute of Industrial Science, The University of Tokyo \\
  Universal Biology Institute, The University of Tokyo \\
}

\begin{document}

\maketitle

\begin{abstract}
We investigate Optimization under Input Uncertainty (OIU), in which the input to the objective function, rather than the objective function itself, is subject to uncertainty.
OIU appears in manufacturing processes with production tolerance, control of physical systems with actuation noise, Mixture of Experts, and Reinforcement Learning (RL).
Most of the existing approaches solve OIU by using the value of the objective function but discard the information of the realized input, even though the realized input is observable in various applications.
The question here is whether the discarded information of the realized input is useful to accelerate the optimization process.
We affirmatively answer this question for Evolutionary Strategy (ES) by theoretically showing that the information of the realized input can reduce the variance of the gradient estimator via Rao-Blackwellization.
Using the Rao-Blackwellized gradient estimator, we propose Phenotype-Accelerated Evolutionary Strategy (PAES), which is a refinement of ES for OIU.
Numerical experiments show that PAES converges faster than the usual ES from simple continuous optimization problems to RL benchmarks.
\end{abstract}


\section{Introduction}
In this paper, we investigate a class of stochastic optimization problems in which the input to the objective function, rather than to the objective function itself, is subject to uncertainty.
Concretely, we specify an input distribution $\sampleDist{\theta}(x)$ of the input $x$ parameterized by $\theta$ and then observe the realization $x$ of the input and its evaluation $f(x)$.
Our aim is to find the optimal parameter $\theta$ of the input distribution that maximizes 
\begin{align}
    \baseObjFunc(\theta) := \E{\sampleDist{\theta}(x)}{f(x)},
\end{align}
the expected value of the evaluation.
We call this problem \emph{Optimization under Input Uncertainty} (OIU) (Figure~\ref{fig:oiu-schematic}).
\begin{figure}[t]
    \centering
    \includegraphics[page=1, trim=30 430 30 30, clip, width=0.8\textwidth]{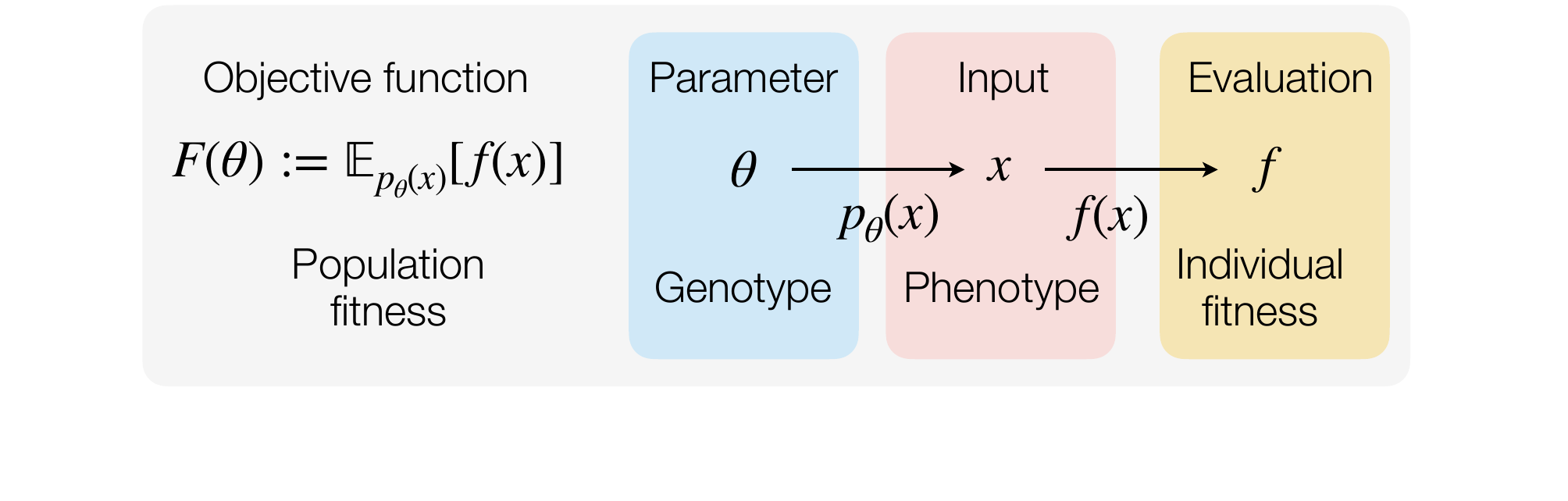}
    \caption{Schematic illustration of Optimization under Input Uncertainty (OIU) and its biological analogy. The parameter $\theta$ (genotype) stochastically determines the input $x$ (phenotype) via $\sampleDist{\theta}(x)$, which is then evaluated by $f$ (individual fitness). The objective $\baseObjFunc(\theta) = \E{\sampleDist{\theta}(x)}{f(x)}$ corresponds to the population fitness with genotype $\theta$.}
    \label{fig:oiu-schematic}
\end{figure}

The uncertainty of input in optimization problems can be found in various contexts~\cite{beyer2007robust,yaochu2005evolutionary}, such as the manufacturing process with production tolerance~\cite{beyer2007robust, Taguchi1989quality,chen1996procedure}, the control of robots and spacecraft with actuation noise~\cite{frohlich2020noisyinput,ollivier2017information}, and performative prediction in machine learning, where the distribution of data depends on the model parameters~\cite{perdomo2020performative}.
While these examples regard randomness in $x$ as an externally imposed nuisance, the same setup also appears when an algorithm designer intentionally introduces randomness.
In addition, we sometimes employ a distribution $\sampleDist{\theta}(x)$ when $x$ is discrete to make $\E{\sampleDist{\theta}(x)}{f(x)}$ differentiable with respect to $\theta$.
Mixture of Experts (MoE)~\cite{jordan1993hierarchical} and classification problems are representative examples.
Another important class of problems is Reinforcement Learning (RL)~\cite{sutton1998reinforcement}, where we introduce stochasticity to a policy to encourage exploration for finding better actions.
We also note that a similar structure appears in prompt engineering, since the outputs of Large Language Models (LLMs) are stochastic.

Since a variety of important real-world problems are categorized into this class of problems, several algorithms have been proposed to deal with the uncertainty of the input.
An approach is gradient-based algorithms, which minimize the worst-case objective value by assuming boundedness of uncertainty~\cite{bertsimas2010robust, Hughes2019largest, hughes2020particle}.
The second class is Bayesian optimization~\cite{frohlich2020noisyinput, oliveira2019bayesian}, which constructs a prior distribution of the objective function by incorporating the uncertainty of the input and then updates the distribution by observing the realized input and its evaluation.
Yet another approach, which we focus mainly on in this paper, is population-based optimization~\cite{beyer2004actuator}, such as Evolutionary Strategy (ES)~\cite{rechenberg1973evolutionary} and genetic algorithms~\cite{tsutsui1999comparative}.
ES has demonstrated its efficiency and superiority in RL~\cite{salimans2017evolution,mania2018simple} in terms of scalability and robustness to local optima compared to single-agent RL algorithms~\cite{lehman2018es,conti2018improving} (See Appendix~\ref{sec:related-work} for more detail).

While these algorithms search for an optimal $\theta$ using $f(x)$ of the realized input $x$ as a feedback signal, they do not utilize the information of the realized input $x$ itself.
Since the realized input $x$ is observable in many applications, the information it carries could be used to further improve the efficiency of the algorithms. 
In the manufacturing process, for example, the production tolerance prevents us from completely controlling the resulting product $x$, but the manufactured product $x$ can be measured and therefore observable.
Likewise, in the control of physical systems, actuation noise hinders us from completely controlling the trajectory of the system, but we can measure the trajectory afterwards.
In particular, ES applied to RL~\cite{salimans2017evolution} uses only the reward of the episode without tapping into the trajectory of states and actions.

These examples raise a question: \emph{is the observation of the input useful for accelerating the optimization?}
Motivated by the analogy with evolutionary biology, we introduce the \textbf{\emph{Phenotype-Accelerated Evolutionary Strategy (PAES)}} which achieves the acceleration of ES. 
In the context of evolutionary biology, the function $f(x)$ is fitness, which is determined by the phenotypic state $x$ of an individual (Figure~\ref{fig:oiu-schematic}).
The phenotype $x$ is stochastically determined by the genotype $\theta$ as $p_{\theta}(x)$. Thus, the objective function $\baseObjFunc(\theta)$ is the fitness of a population with genotype $\theta$. Following the idea of accelerating evolution by learning at the phenotypic level~\cite{nakashima2022acceleration,nakashima2025unifying}, PAES takes advantage of the phenotypic information to design a population-based efficient gradient estimator.
Furthermore, we affirmatively and theoretically answer the question by showing that this acceleration is Rao-Blackwellization~\cite{rao1945information, blackwell1947conditional,lehmann1998theory}.
Rao-Blackwellization is a technique to reduce the variance of estimation in which we eliminate the randomness from irrelevant random variables by taking a conditional expectation (See Appendix~\ref{sec:related-work}).
PAES replaces the gradient estimation of ES for OIU with its Rao-Blackwellization.

The contribution of this paper is summarized as follows:
\begin{itemize}
\item We introduce a general framework for optimization under input uncertainty (OIU).
\item We show theoretically that observing the realized input can accelerate ES for OIU through Rao--Blackwellization, and we call the resulting accelerated method PAES.
\item For a class of problems, we derive the Rao--Blackwellized gradient estimator analytically and provide a reparameterization-trick approximation to it~\cite{kingma2014auto,rezende2014stochastic}.
\item We demonstrate numerically that PAES converges faster than ES across benchmarks, from simple optimization problems to more complex RL tasks.
\end{itemize}

We note that a Rao-Blackwellization of a gradient estimator is employed in ~\cite{lam2026raoblackwellised}. Yet, this work differs from ours in that they take a conditional expectation of the latent variable with respect to the activation of hidden layer whereas we take a conditional expectation of parameters with respect to the realized input.

\subsection{Organization of the Paper}
The rest of the paper is organized as follows.
In Section~\ref{sec:preliminaries}, we explain notation and preliminaries.
In Section~\ref{sec:problem-setting}, we formulate OIU and outline some examples of OIU.
In Section~\ref{sec:algorithm}, we introduce PAES and show that the variance of the gradient estimator of PAES is smaller than that of ES via Rao-Blackwellization.
We also give an analytical form of the gradient estimator of PAES for some specific cases and device an approximation of the gradient estimator via the reparameterization trick.
In Section~\ref{sec:example-of-paes} and Section~\ref{sec:numerical-experiments}, we demonstrate PAES with some examples and conduct the numerical experiments of PAES, respectively.
In Section~\ref{sec:conclusion}, we conclude the paper.
Appendix contains the omitted proofs from the main text and additional experimental results.

\section{Preliminaries}
\label{sec:preliminaries}
Before introducing PAES, we clarify the notation to be used and outline the basis of ES and its variants.

\subsection{Notation}
For a positive integer $n$, let $[n] = \{1, 2, \ldots, n\}$.
For simplicity, we sometimes denote summation $\sum_{i=1}^N f(x_i)$ as $\int f(x) \dd x$.
We denote the Kronecker delta function as $\delta_{x, x'}$.
For a vector $\bm{x} = (x_1, x_2, \ldots, x_n)$, its diagonalization $\mathrm{diag}(\bm{x})$ is defined as the diagonal matrix with $x_1, x_2, \ldots, x_n$ on its diagonal.

\subsection{Random Variables and Probability Distributions}
When a random variable $x$ follows a certain distribution whose probability density function is $p(x)$, we write $x \sim p(x)$. $\E{p(x)}{f(x)}$ denotes the expectation of $f(x)$ with respect to $p(x)$. 
For a random vector $v \sim p(v)$, we define the variance of $v$ as $\V{p(v)}{v} = \E{p(v)}{\|v\|^2}$.

For a $n$-dimensional normal distribution, we denote its probability density function by $\normal{x}{\mu}{\Sigma}$, where $x \in \R^n$ is a random variable, $\mu \in \R^n$ is the mean vector, and $\Sigma \in \R^{n\times n}$ is the covariance matrix.
Similarly, for a Dirichlet distribution, we denote its probability density function by $\mathrm{Dir}(\pi \mid \alpha)$:
\begin{align}
    \mathrm{Dir}(\pi \mid \alpha)
    &:= \frac{1}{\mathcal{B}(\alpha)} \prod_{i \in [n]} \pi_{i}^{\alpha_{i} - 1},
\end{align}
where $\pi$ is a probability distribution on $[n]$ and  and $\alpha\in \R^n_{>0}$ is the parameter.
Here, $\mathcal{B}(\alpha)$ is the extended Beta function defined by
\begin{align}
    \mathcal{B}(\alpha) := \frac{\prod_{i \in [n]} \Gamma(\alpha_{i})}{\Gamma\left(\sum_{i \in [n]} \alpha_{i}\right)},
\end{align}
where $\Gamma(x)$ is the Gamma function.

\subsection{Evolutionary Strategy}
Given an unconstrained maximization problem of an objective function $F \colon \SearchSp \to \R$, where $\SearchSp \subseteq \R^n$ is a search space and its element is typically represented by $\theta$, ES solves the problem by introducing a population, which is written as a distribution $\searchDist{\psi}$ on $\SearchSp$ parameterized by $\psi$.
Instead of directly optimizing $F(\theta)$, ES seeks a parameter $\psi$ that maximizes the expected value of $F(\theta)$ with respect to $\searchDist{\psi}(\theta)$~\cite{rechenberg1973evolutionary}:
\begin{align}
    \label{eq:def-smoothe-obj}
    \smoothedObjFunc(\psi) := \E{\searchDist{\psi}(\theta)}{\baseObjFunc(\theta)}.
\end{align}
This expected value is maximized by a stochastic gradient ascent method:
\begin{align}
    \label{eq:es-update-rule}
    \psi_{t+1} = \psi_t + \eta \overline{\frac{\partial \smoothedObjFunc}{\partial \psi}},
\end{align}
where $\overline{\frac{\partial \smoothedObjFunc}{\partial \psi}}$ is an estimator of $\frac{\partial \smoothedObjFunc}{\partial \psi}$ and $\eta$ is a learning rate.
By applying the log-derivative trick, the gradient is calculated without computing $\partial F(\theta)/\partial \theta$ as follows
\begin{align}
    \label{eq:gradient-log-derivative-without-transformation}
    \frac{\partial \smoothedObjFunc}{\partial \psi} = \E{\searchDist{\psi}(\theta)}{F(\theta)  \frac{\partial \log \searchDist{\psi}(\theta)}{\partial \psi}}.
\end{align}
From this equation, we can estimate the gradient using samples $\theta_i \sim \searchDist{\psi}$ for $i \in [\popSize]$:
\begin{align}
    \frac{\partial \smoothedObjFunc}{\partial \psi} \approx \frac{1}{\popSize} \sum_{i=1}^{\popSize} {F(\theta_i)}  \frac{\partial \log \searchDist{\psi}(\theta_i)}{\partial \psi} =: \overline{\frac{\partial \smoothedObjFunc}{\partial \psi}}.
\end{align}

\subsection{Modifications and Refinements of ES}
The gradient estimator of ES~\eqref{eq:gradient-log-derivative-without-transformation} is sensitive to the scale of the objective function $F(\theta)$, which may worsen the performance of ES in some cases.
To address this issue, we can introduce a transformation function $W_{\{F(\theta_i)\}_i}: \R \to \R$, which can depend on a set of evaluations $\{F(\theta_i)\}_i$ over the populations.
We define a transformed objective function by
\begin{align}
    \smoothedObjFunc_{W}(\psi) := \E{\searchDist{\psi}(\theta)}{W(F(\theta))},
\end{align}
and the associated gradient estimator from samples~\cite{ollivier2017information};
\begin{align}
    \label{eq:gradient-approx}
    \frac{\partial \smoothedObjFunc_{W}}{\partial \psi} \approx \frac{1}{\popSize} \sum_{i=1}^{\popSize} {W(F(\theta_i))}  \frac{\partial \log \searchDist{\psi}(\theta_i)}{\partial \psi},
\end{align}
where we omit the subscript of $W_{\{F(\theta_i)\}_i}$ for notational simplicity.
We can further improve the performance of ES by introducing a natural gradient estimator.
Using the Fisher information matrix
\begin{align}
    \label{eq:fisher-information-matrix-search-dist}
    \FisherInfo{\searchDist{\psi}(\theta)} := \E{\searchDist{\psi}(\theta)}{ \frac{\partial \log \searchDist{\psi}(\theta)}{\partial \psi} \frac{\partial \log \searchDist{\psi}(\theta)}{\partial \psi}^{\top}},
\end{align}
the natural gradient is obtained; 
\begin{align}
    \nabla \smoothedObjFunc_{W}(\psi) = \FisherInfo{\searchDist{\psi}(\theta)}^{-1} \E{\searchDist{\psi}(\theta)}{ W(F(\theta))   \frac{\partial \log \searchDist{\psi}(\theta)}{\partial \psi}},
\end{align}
where and hereafter $\nabla$ instead of $\partial/\partial \phi$ is used to indicate the natural gradient.
We can approximate the natural gradient by samples in a similar way to \eqref{eq:gradient-approx}.
This method is known as Information Geometric Optimization~\cite{ollivier2017information}.
The algorithm with transformation function $W$ is summarized in Algorithm~\ref{alg:ES-W}.

\begin{algorithm}
    \caption{Evolutionary Strategy (ES) with Transformation Function $W$ }
    \label{alg:ES-W}
    \begin{algorithmic}
        \State \textbf{Require:} Initial parameter $\psi_0$, population size $\popSize$, learning rate $\eta$, stopping criterion
        \State \textbf{Input:} Objective function $F$, search distribution $\searchDist{\psi}$, transformation function $W$
        \State \textbf{Output:} Optimized parameter $\psi_t$
        \State $t \gets 0$.
        \While{Stopping criterion is not met}
            \For{$i \in [\popSize]$}
                \State Sample $\theta_i \sim \searchDist{\psi_t}$.
                \State Evaluate $F(\theta_i)$.
            \EndFor
            \State Calculate gradient estimator $\overline{\frac{\partial \smoothedObjFunc_{W}}{\partial \psi}} = \frac{1}{\popSize} \sum_{i=1}^{\popSize} {W(F(\theta_i))}  \frac{\partial \log \searchDist{\psi_t}(\theta_i)}{\partial \psi}$.
            \State Update parameter by $\psi_{t+1} \gets \psi_t + \eta \overline{\frac{\partial \smoothedObjFunc_{W}}{\partial \psi}}$ or $\psi_{t+1} \gets \psi_t + \eta \FisherInfo{\searchDist{\psi_t}(\theta)}^{-1} \overline{\frac{\partial \smoothedObjFunc_{W}}{\partial \psi}}$.
            \State $t \gets t+1$
        \EndWhile
    \end{algorithmic}
\end{algorithm} 

There exist other finer modifications of the gradient estimator~\eqref{eq:gradient-approx}.
The two-point gradient estimator~\cite{shamir2017optimal,flaxman2005online,nesterov2017random,duchi2015optimal} is a major example, and its theoretical analysis shows the superiority of the two-point estimator over the one-point gradient estimator~\eqref{eq:gradient-approx}.
The two-point gradient estimator for general ES is defined as follows:
\begin{align}
    \label{eq:gradient-approx-two-point}
    \overline{\frac{\partial \smoothedObjFunc_{W}}{\partial \psi}} := 
    \frac{1}{\popSize} \sum_{i=1}^{\popSize} \left[ \frac{W(F(\theta^i_1))-W(F(\theta^i_2))}{2\sigma} \right]\left[ \frac{\partial \log \searchDist{\psi}(\theta^i_1)}{\partial \psi} - \frac{\partial \log \searchDist{\psi}(\theta^i_2)}{\partial \psi} \right].
\end{align}

\subsection{Zeroth-Order Optimization: an instance of ES}
One of the most popular versions of ES is the case in which the population $\searchDist{\psi}(\theta)$ is a Gaussian distribution $\normal{\theta}{\bar \theta}{\sigma^2 \identityMatrix}$ with $\sigma$  as a hyperparameter.
In this case, the gradient of $\smoothedObjFunc_{W}$  in~\eqref{eq:gradient-approx} is reduced to the following form:
\begin{align}
    \label{eq:gradient-approx-zeroth-order}
    \frac{\partial \smoothedObjFunc_{W}}{\partial \psi} \approx \frac{1}{\popSize \sigma} \sum_{i=1}^{\popSize} {W(F(\theta_i))}  \epsilon_i ,
\end{align}
where $\epsilon_i = (\theta_i - \bar \theta) / \sigma$.
This method is called Zeroth-Order Optimization (ZOO).
The two-point gradient estimator for ZOO becomes 
\begin{align}
    \overline{\frac{\partial \smoothedObjFunc_{W}}{\partial \psi}}  = 
    \frac{1}{\popSize\sigma} \sum_{i=1}^{\popSize} \left[ \frac{W(F(\theta^i_1))-W(F(\theta^i_2))}{2\sigma} \right] \epsilon_i.
\end{align}

\section{Formalization of Optimization under Input Uncertainty (OIU)}
\label{sec:problem-setting}

In this section, we formalize OIU.
Let $\SampleSp \subset \R^n$ be an input space (possibly a discrete space), and $f \colon \SampleSp \to \R$ be an objective function over the input space.
Let $\SearchSp \subseteq \R^m$ and assume that, for each $\theta \in \SearchSp$, a probability distribution $\sampleDist{\theta}$ on $\SampleSp$ is associated.
Our aim is to maximize the following objective function over the search space $\SearchSp$:
\begin{align}
    \label{eq:objective-func}
    \baseObjFunc(\theta) := \E{\sampleDist{\theta}(x)}{f(x)}.
\end{align}
In the following, we call $\sampleDist{\theta}$ the \emph{input distribution}.

\subsection{Examples of OIU}
\label{subsubsec:example-problem}
\subsubsection{Continuous Optimization with Input Noise}
\label{subsubsec:simple-example-of-continuous-optimization-with-input-noise}
Let us consider the case where we optimize an objective function $f \colon \R^n \to \R$, but the input $x$ is perturbed by an input noise.
Specifically, in each trial, we choose an intended input $ \theta_i \in \R^n$.
Then, we observe a realized input $x_i$, which is perturbed by the input noise as $x_i = \theta_i + \epsilon_i$, where $\epsilon_i$ is the input noise that follows some distribution $h(\epsilon_i \mid \theta_i)$.
We then observe the feedback $f(x_i)$.
In this situation, our objective is to maximize the expected value of $\E{\epsilon_i}{f(\theta_i + \epsilon_i)}$.
By setting $\sampleDist{\theta_i}(x) = h(x - \theta_i \mid \theta_i)$, this problem is formulated within the framework of~\eqref{eq:objective-func}.

\subsubsection{Reinforcement Learning}
\label{subsubsec:reinforcement-learning}
In RL, our objective is to find a policy $\pi_{\theta}(a \mid s)$, where $a$ is an action and $s$ is a state, which maximizes the expected cumulative reward.
The expected cumulative reward is given by
\begin{align}
    \baseObjFunc(\theta) = \E{\forward^{\theta}}{\sum_{t=0}^\infty \gamma^t r(s_t, a_t)},
\end{align}
where $r(s, a)$ is a reward function and $\gamma \in [0,1)$ is a discount factor.
Here, we define $\forward^{\theta}$ as the law of the trajectory of states and actions produced by a policy $\pi_{\theta}(a \mid s)$ and a given transition matrix $\transitMatrix(s' \mid s, a)$.
The explicit form of $\forward^{\theta}$ is 
\begin{align}
    \forward^{\theta}(a_0, s_1, a_1, s_2, \dots \mid s_0) = \prod_{t=0}^\infty \pi_{\theta}(a_t \mid s_t) \transitMatrix(s_{t+1} \mid s_t, a_t).
\end{align}
By setting sample $x = (a_0, s_1, a_1, s_2, \dots)$ and $\sampleDist{\theta} = \forward^{\theta}$, we see that this problem is in the framework of~\eqref{eq:objective-func}.

\subsection{Extension of OIU}
We may extend OIU to Contextual Optimization under Input Uncertainty (COIU) where $f(x)$ and $p_{\theta}(x)$ are also dependent on a context $c$ as in $f(x,c)$ and $p_{\theta}(x|c)$. 
This extension covers a wider range of applications such as mixture of experts and classification with bandit feedback.
 To keep the main body of the text simple, we discuss this extension in Appendix \ref{sec:COIU}.

\section{Phenotype-Accelerated Evolutionary Strategy (PAES) for OIU}
\label{sec:algorithm}
In this section, we introduce PAES and its core component, the phenotype gradient estimator. As PAES is an extension of ES for OIU, we first illustrate how ES is adopted to solve OIU and how its gradient estimator can be improved by exploiting the realized input $x$.

\subsection{Evolutionary Strategy for OIU and Genotype Gradient Estimator}
Give the objective function $F(\theta)$ of OIU, its smoothed version by a distribution $\searchDist{\psi}(\theta)$ is obtained as
\begin{align}
    \label{eq:smoothe-obj-oiu}
    \smoothedObjFunc(\psi) := \E{\searchDist{\psi}(\theta)}{\baseObjFunc(\theta)} = \E{\searchDist{\psi}(\theta)}{\E{\sampleDist{\theta}(x)}{f(x)}}.
\end{align}
where $\searchDist{\psi}(\theta)$ is interpreted as the diversity of genotypes (Figure~\ref{fig:paes-schematic}).
We can also define a transformed objective function:
\begin{align}
    \smoothedObjFunc_{W}(\psi) := \E{\searchDist{\psi}(\theta)}{\E{\sampleDist{\theta}(x)}{W(f(x))}}.
\end{align}
The gradient is calculated via~\eqref{eq:gradient-log-derivative-without-transformation}:
\begin{align}
    \label{eq:gradient-log-derivative-input-noise}
    \frac{\partial \smoothedObjFunc_{W}}{\partial \psi} = \E{\searchDist{\psi}(\theta)}{\E{\sampleDist{\theta}(x)}{W(f(x))}  \frac{\partial \log \searchDist{\psi}(\theta)}{\partial \psi}}.
\end{align}
Then, we obtain an estimate of this expression by sampling:
\begin{definition}[Genotype gradient estimator]
    \begin{align}
        \label{eq:genotype-gradient-estimator}
        \genotypeEstimator \smoothedObjFunc_{W} := \frac{1}{\popSize} \sum_{i=1}^{\popSize} {W(f(x_i)) \frac{\partial \log \searchDist{\psi}(\theta_i)}{\partial \psi}}.
    \end{align}
Since this estimator is obtained by applying ES to the objective function $F(\theta)$ with respect to $\theta$, i.e., genotype information, we call it the genotype gradient estimator.
\end{definition}
We can also apply ZOO~\eqref{eq:gradient-approx-zeroth-order} to continuous optimization with input noise (Section \ref{subsubsec:simple-example-of-continuous-optimization-with-input-noise}), for which the genotype gradient estimator is given by
\begin{align}
    \label{eq:gradient-approx-zeroth-order-input-noise}
    \frac{\partial \smoothedObjFunc_{W}}{\partial \psi} \approx \frac{1}{\popSize \sigma} \sum_{i=1}^{\popSize} {W(f(x_i))}  \epsilon_i,
\end{align}
where $\epsilon_i = (\theta_i - \bar \theta) / \sigma$.

\subsection{Intuition behind Improvement of Gradient by Phenotypic Information}
One potential problem of the genotype gradient estimator~\eqref{eq:gradient-approx-zeroth-order-input-noise} is that it utilizes only the evaluation $f(x_i)$ and the intended input $\theta_i$, i.e., the fitness value and the gennotype information.
The information of realized input $x_i$, i.e., the phenotypic information, is discarded.
Let us illustrate that the realized input $x_i$ has richer information than $f(x_i)$ and $\theta_i$ for gradient estimation by the following example.

We consider an extreme example under the problem setting in Section~\ref{subsubsec:simple-example-of-continuous-optimization-with-input-noise}, where the population distribution is singular.
Concretely, we assume that $h(\epsilon_i \mid \theta_i) = \normal{\epsilon_i}{\bm{0}}{\identityMatrix}$, where $\epsilon_i = (x_i - \theta_i) / \sigma$.
We also assume that $\psi = \theta$ and $\searchDist{\psi}$ is a delta distribution $\delta_{\theta}$.
We further assume that $W(f(x)) = f(x)$.
As a consequence, the genotype gradient estimator~\eqref{eq:gradient-approx-zeroth-order-input-noise} becomes meaningless since we can no longer define $\frac{\partial \log \searchDist{\psi}(\theta)}{\partial \psi}$.
This is because, while the estimator~\eqref{eq:gradient-approx-zeroth-order-input-noise} uses the diversity of $\theta$ for the estimation of the gradient, there is no diversity of $\theta$ in this case.
However, by following a similar argument to the usual ES, we can estimate the gradient even in this case by using the realized input.
As shown in the next section (see Theorem~\ref{thm:full-gradient-estimator}), we have
\begin{align}
    \label{eq:ancestral-gradient-estimation}
    \frac{\partial \smoothedObjFunc_{W}}{\partial \psi} = \frac{1}{\sigma} \E{\sampleDist{\theta}(x)}{f(x)\epsilon} = \frac{1}{\sigma^2} \E{\sampleDist{\theta}(x)}{f(x)(x - \theta)},
\end{align}
where $\epsilon = (x - \theta) / \sigma$.
From this equation, we can estimate the gradient as follows:
\begin{align}
    \frac{\partial \smoothedObjFunc_{W}}{\partial \theta} \approx \frac{1}{\popSize \sigma^2} \sum_{i=1}^{\popSize} {f(x_i)(x_i - \theta)}.
\end{align}

This example illuminates that the realized input $x_i$ is useful for estimating gradient, and we can construct a new gradient estimator~\eqref{eq:ancestral-gradient-estimation}.

\subsection{Phenotype gradient estimator and PAES}
By generalizing the procedure in the example above, we define another gradient estimator:
\begin{definition}[Phenotype gradient estimator]
For a joint distribution $\forward(x, \theta) = \sampleDist{\theta}(x) \searchDist{\psi}(\theta)$ and its marginal distribution $\forward(x) = \int \forward(x, \theta)  \dd \theta$, a gradient estimator is defined as follows:
    \begin{align}
        \label{eq:full-gradient-estimator}
        \fullGradientEstimator \smoothedObjFunc_{W} := \frac{1}{\popSize} \sum_{i=1}^{\popSize} {W(f(x_i)) \frac{\partial \log \forward(x_i)}{\partial \psi}}.
    \end{align}
We call this estimator \emph{phenotype gradient estimator} since it utilizes the diversity of both genotype $\theta$ and phenotype $x$.
\end{definition}
The relationship between the genotype gradient estimator and the phenotype gradient estimator is illustrated in Figure~\ref{fig:paes-schematic}.
\begin{figure}[t]
    \centering
    \includegraphics[page=2, trim=30 160 30 30, clip, width=\textwidth]{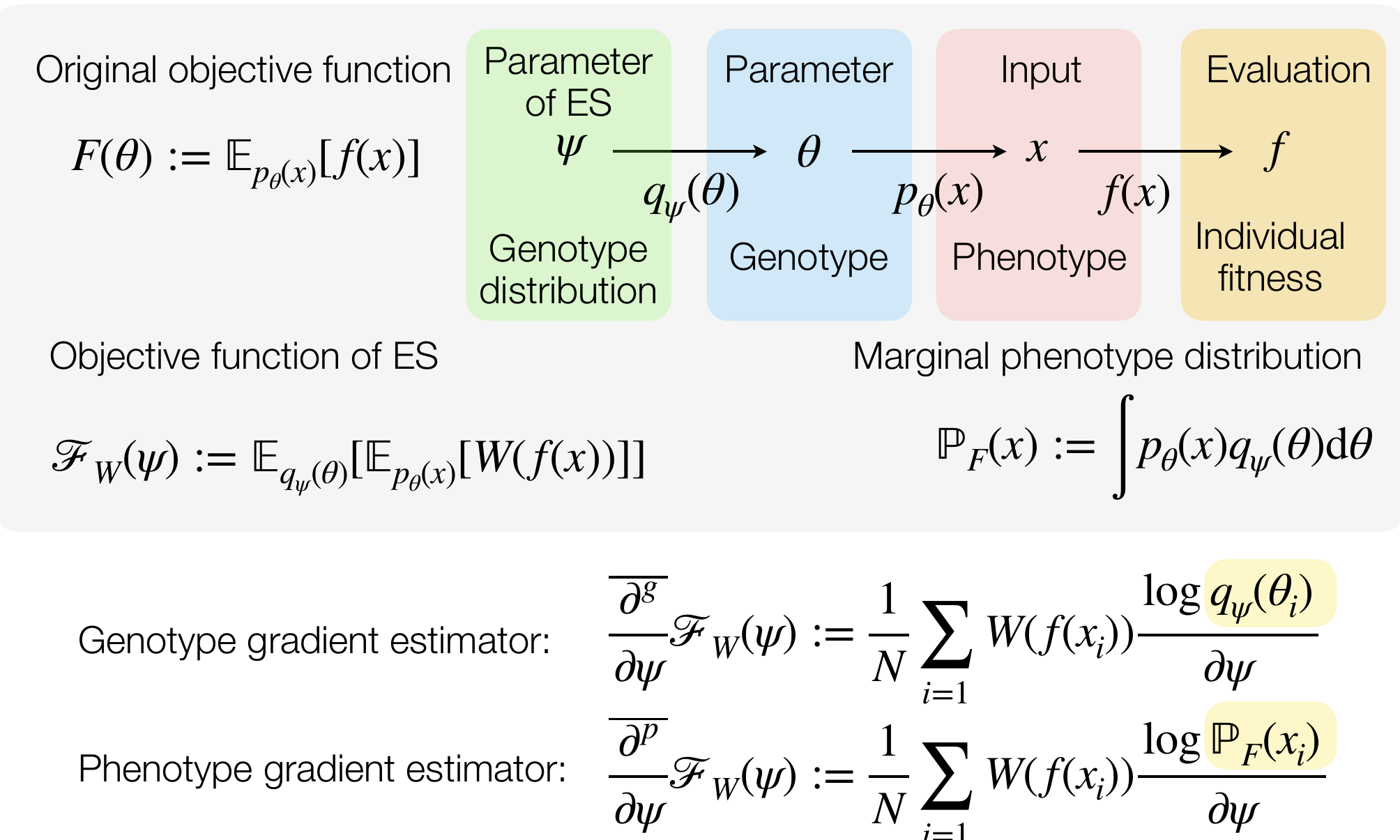}
    \caption{Schematic illustration of the application of ES to OIU and its refinement, phenotype-accelerated evolutionary strategy (PAES). We introduce $\searchDist{\psi}(\theta)$ as the population of ES, which we regard as the diversity of genotypes in the context of evolutionary biology.
    We can estimate the gradient by using only the genotypic information $\theta_i$ and the evaluation $f(x_i)$, which is called the genotype gradient estimator. However, we can also utilize the phenotypic information $x_i$ to improve the estimator, which we call the phenotype gradient estimator.
    }
    \label{fig:paes-schematic}
\end{figure}
\begin{theorem}
The phenotype gradient estimator is an unbiased estimator as the genotype gradient estimator is:
    \label{thm:full-gradient-estimator}
    \begin{align}
        \label{eq:full-gradient}
        \frac{\partial \smoothedObjFunc_{W}}{\partial \psi} = \E{\forward(x)}{W(f(x))\frac{\partial \log \forward(x)}{\partial \psi}}.
    \end{align}
In addition, the phenotype gradient estimator is superior to the genotype gradient estimator in that it has smaller variance than the genotype gradient estimator:
    \label{thm:genotype-gradient-estimator-property}
    \begin{align}
        \label{eq:genotype-gradient-estimator-property}
        \V{\forward}{\fullGradientEstimator \smoothedObjFunc_{W}} \leq \V{\forward}{\genotypeEstimator \smoothedObjFunc_{W}}.
    \end{align}
\end{theorem}

\begin{proof}[Proof of Equation~\eqref{eq:full-gradient} and Equation~\eqref{eq:ancestral-gradient-estimation}]
    We have
    \begin{align}
        \frac{\partial \smoothedObjFunc_{W}}{\partial \psi} 
        &= \frac{\partial}{\partial \psi} \int W(f(x)) \forward(x) \dd x \\
        &= \int W(f(x)) \frac{\partial \forward(x)}{\partial \psi} \dd x \\
        &= \int  \left[  W(f(x))\frac{\partial \log \forward(x)}{\partial \psi} \right]   \forward(x)   \dd x \\
        &= \E{\forward(x)}{W(f(x)) \frac{\partial \log \forward(x)}{\partial \psi}}.
    \end{align}
    This proves Equation~\eqref{eq:full-gradient}.
    Equation~\eqref{eq:ancestral-gradient-estimation} is the special case of this equation.
\end{proof}
The proof of Equation~\eqref{eq:genotype-gradient-estimator-property} is given in Section~\ref{subsec:algorithm-proof}.

By replacing the genotype gradient estimator in ES for OIU with the phenotype gradient estimator, we introduce \emph{Phenotype-Accelerated Evolutionary Strategy} (PAES)~(Algorithm~\ref{alg:PAES}) as follows:
\begin{algorithm}
    \caption{Phenotype-Accelerated Evolutionary Strategy (PAES)}
    \label{alg:PAES}
    \begin{algorithmic}
        \State \textbf{Require:} Initial parameter $\suptime{\psi}{0}$, population size $\popSize$, learning rate $\learningRate$, stopping criterion
        \State \textbf{Input:} Objective function $f$, search distribution $\searchDist{\psi}$, sampling distribution $\sampleDist{\theta}$, transformation function $W$
        \State \textbf{Output:} Optimized parameter $\suptime{\psi}{t}$
        \State $t \gets 0$.
        \While{Stopping criterion is not met}
            \For{$i \in [\popSize]$}
                \State Sample $\suptime{\theta_i}{t} \sim \searchDist{\suptime{\psi}{t}}$.
                \State Sample $\suptime{x_i}{t} \sim \sampleDist{\suptime{\theta_i}{t}}$.
            \EndFor
            \State Calculate phenotype gradient estimator $\fullGradientEstimator \smoothedObjFunc_{W}$ via~\eqref{eq:full-gradient-estimator}.
            \State Update $\suptime{\psi}{t+1} \gets \suptime{\psi}{t} + \learningRate  \fullGradientEstimator \smoothedObjFunc_{W}$ or $\suptime{\psi}{t+1} \gets \suptime{\psi}{t} + \learningRate \FisherInfo{\forward(x)}^{-1} \fullGradientEstimator \smoothedObjFunc_{W}$.
            \State $t \gets t+1$.
        \EndWhile
    \end{algorithmic}
\end{algorithm}

\subsection{Phenotype Gradient Estimator as the Rao-Blackwellized Estimator}
To systematically extend generalized genotype gradient estimators (e.g., the two-point estimator~\eqref{eq:gradient-approx-two-point}) to their phenotype counterparts, it is crucial to first generalize Theorem~\ref{thm:genotype-gradient-estimator-property}.
Here, we clarify that this generalization is achieved by an argument similar to the Rao-Blackwell theorem.

Recall that, at each generation $t$, we sample $(\suptime{\theta_i}{t}, \suptime{x_i}{t})$ from $\forward(x, \theta)$ independently for $i \in [\popSize]$.
Let us assume that we have some unbiased estimator $\estimator{\{\suptime{\theta_i}{t}, \suptime{x_i}{t}\}_{i \in [\popSize]}}$ of the gradient for these samples $i \in [\popSize]$.
\begin{definition}[Rao-Blackwellized estimator]
    \label{def:rao-blackwellized-estimator}
    Given an estimator $\estimator{\{\suptime{\theta_i}{t}, \suptime{x_i}{t}\}_{i \in [\popSize]}}$, a Rao-Blackwellized estimator $\RaoBlackwellizeNoArgs{\estimatorChar}$ is defined by
    \begin{align}
        \label{eq:def-rao-blackwellized-estimator}
        \RaoBlackwellize{\estimatorChar}{\{\suptime{x_i}{t}\}_{i \in [\popSize]}} := \E{\forward[\{ \suptime{\theta_i}{t} \}_{i \in [\popSize]} \mid \{\suptime{x_i}{t}\}_{i \in [\popSize]}]}{\estimator{\{\suptime{\theta_i}{t}, \suptime{x_i}{t}\}_{i \in [\popSize]}}},
    \end{align}
    where the expectation is taken over $\forward(\theta \mid x)$ for all pairs of $(\suptime{\theta_i}{t}, \suptime{x_i}{t})$ independently.
    Here, $\forward(\theta \mid x)$ is the conditional distribution defined by $\forward(x, \theta) / \forward(x)$.
\end{definition}

For this Rao-Blackwellized estimator, we can prove a generalization of Theorem~\ref{thm:genotype-gradient-estimator-property} .
\begin{theorem}[Generalization of Theorem~\ref{thm:full-gradient-estimator}]
    The Rao-blackwellized estimator is unbiased.
    \label{thm:generalized-unbiasedness-of-rao-blackwellized-estimator}
    \begin{align}
        \label{eq:generalized-unbiasedness-of-rao-blackwellized-estimator}
        \E{\forward}{\RaoBlackwellize{\estimatorChar}{\{\suptime{x_i}{t}\}_{i \in [\popSize]}}} = \E{\forward}{\estimator{\{\suptime{\theta_i}{t}, \suptime{x_i}{t}\}_{i \in [\popSize]}}}.
    \end{align}
In addition, the Rao-blackwellized estimator has smaller variance than the original estimator:
    \label{thm:generalized-superiority-of-full-gradient-estimator}
    \begin{align}
        \label{eq:generalized-variance-reduction-of-rao-blackwellized-estimator}
        \V{\forward}{\RaoBlackwellize{\estimatorChar}{{\{\suptime{x_i}{t}\}_{i \in [\popSize]}}}} \leq \V{\forward}{\estimator{\{\suptime{\theta_i}{t}, \suptime{x_i}{t}\}_{i \in [\popSize]}}}.
    \end{align}
\end{theorem}
The proofs are given in Section~\ref{subsec:algorithm-proof}.
We can check that Theorem~\ref{thm:genotype-gradient-estimator-property} is a special case of Rao-Blackwellization by using the following proposition:
\begin{proposition}[Phenotype gradient estimator is a Rao-Blackwellized estimator]
    \label{thm:rao-blackwellized-genotype-gradient-estimator}
    \begin{align}
        \RaoBlackwellizeNoArgs{\genotypeEstimator \smoothedObjFunc_{W}} = \fullGradientEstimator \smoothedObjFunc_{W}.
    \end{align}
\end{proposition}
The proof is given in Section~\ref{subsec:algorithm-proof}.

\subsection{Analytical Form of Phenotype Gradient Estimator}
\label{sec:analytical-form-of-full-gradient-estimator}
The phenotype gradient estimator requires the computation of $\frac{\partial \log \forward(x)}{\partial \psi}$.
This can be analytically achieved when the input distribution $\sampleDist{\theta}$ belongs to an exponential family and $\searchDist{\psi}$ belongs to its conjugate family.

Suppose that $x \in \SampleSp \subseteq \R^k$ and
\begin{align}
    \sampleDist{\theta}(x) = \baseMeasureExpFamilyChar(x)\exp \left( \theta \cdot  \sufficientStatistic{x}  - A(\theta) \right),
\end{align}
where $\sufficientStatistic{x} \colon \SampleSp \to \R^n$ and $h \colon \SampleSp \to  \R$ , and $A(\theta)$ is a logarithm of the normalization constant defined by
\begin{align}
    A(\theta) = \log \int  \baseMeasureExpFamilyChar(x)\exp \left( \theta \cdot  \sufficientStatistic{x}  \right) \dd x.
\end{align}
We also assume that $\searchDist{\psi}$ belongs to the conjugate family of $\sampleDist{\theta}$:
\begin{align}
    \searchDist{\psi}(\theta) = \exp \left( \nu \phi \cdot \theta -\nu A(\theta) - B(\phi, \nu) \right),
\end{align}
where $\nu \in \R$, $\psi = (\phi, \nu) \in \R^k \times \R$, and $B(\phi, \nu)$ is the logarithm of the normalization constant defined by
\begin{align}
    B(\phi, \nu) = \log \int  \exp \left( \nu \phi \cdot \theta -\nu A(\theta) \right) \dd \theta.
\end{align}
\begin{proposition}
For the exponential family and its conjugate, $\frac{\partial \log \forward(x)}{\partial \psi}$ is analytically represented as
\begin{align}
    \label{eq:analytical-form-of-full-gradient-estimator-for-exponential-family}
    &\frac{\partial \log \forward(x)}{\partial \phi} = \E{\searchDist{\psi'}(\theta)}{\nu' \theta} - \E{\searchDist{\psi}(\theta)}{\nu \theta},\\
    &\frac{\partial \log \forward(x)}{\partial \nu} = \E{\searchDist{\psi'}(\theta)}{\phi' \cdot \theta - A(\theta)} - \E{\searchDist{\psi}(\theta)}{\phi \cdot \theta - A(\theta)}.
\end{align}
where $\psi' = (\phi', \nu')$ and
\begin{align}
    &\phi' = \frac{\nu \phi + \sufficientStatistic{x}}{\nu + 1},\\
    &\nu' = \nu + 1.
\end{align}
Moreover, if we parameterize $\searchDist{\psi}$ by the expectation parameter,
\begin{align}
    \eta(\psi) = (\E{\searchDist{\psi}(\theta)}{\nu \theta}, \E{\searchDist{\psi}(\theta)}{\phi \cdot \theta - A(\theta)})^{\transpose},
\end{align}
the natural gradient is calculated as follows~\cite{ollivier2017information}:
\begin{align}
    \label{eq:natural-gradient-for-exponential-family}
    &\nabla_{\eta_{\phi}} \log \forward(x) = \E{\searchDist{\psi'}(\theta)}{\nu' \theta} - \E{\searchDist{\psi}(\theta)}{\nu \theta},\\
    &\nabla_{\eta_{\nu}} \log \forward(x) = \E{\searchDist{\psi'}(\theta)}{\phi' \cdot \theta - A(\theta)} - \E{\searchDist{\psi}(\theta)}{\phi \cdot \theta - A(\theta)}.
\end{align}
The proof is given in Section~\ref{subsec:analytical-proof}.
\end{proposition}

\subsection{Remark on the Space Where Fisher Information Matrix Is Defined}
In equations~\eqref{eq:natural-gradient-for-exponential-family}, the Fisher information matrix is defined by~\eqref{eq:fisher-information-matrix-search-dist}.
However, we have another choice of the Fisher information matrix based on $\forward$:
\begin{align}
    \FisherInfo{\forward(x)} = \E{\forward(x)}{\frac{\partial \log \forward(x)}{\partial \psi} \frac{\partial \log \forward(x)}{\partial \psi}^{\transpose}}.
\end{align}
For some cases (e.g. Section~\ref{subsec:gaussian-search-and-sampling}), this Fisher information matrix leads to an easier update rule.
Let us compare these two definitions of the natural gradient. As discussed in~\cite{amari2016information}, the natural gradient is characterized by the following variational problem.
For any distribution $p$, we have
\begin{align}
    \FisherInfo{p}^{-1} \frac{\partial \smoothedObjFunc_{W}}{\partial \psi} = \lim_{\epsilon \to 0} \argmax_{\delta p} \left[ \smoothedObjFunc_{W}(\psi + \delta p)  \right],
\end{align}
where $\delta p$ satisfies $\KL(p \mid p + \delta p) = \epsilon$.
Thus, the natural gradient is the direction where the objective function increases the most when we measure the "distance" by the KL divergence.
Under this interpretation, the natural gradient with $\FisherInfo{\searchDist{\psi}(\theta)}$ is the direction to which the objective function increases the most when the distance is measured by $\KL(\searchDist{\psi}(\theta) \mid \searchDist{\psi + \delta \psi}(\theta))$.
On the other hand, the natural gradient with $\FisherInfo{\forward(x)}$ is the direction to which the objective function increases the most when the distance is measured by $\KL(\forward(x \mid \psi) \mid \forward(x \mid \psi + \delta \psi))$.
We also note that the Fisher information matrix of $\forward(x, \theta)$ is the same as that of $\searchDist{\psi}(\theta)$:
\begin{align}
    \FisherInfo{\forward(x, \theta)} = \E{\forward(x, \theta)}{\frac{\partial \log \forward(x, \theta)}{\partial \psi} \frac{\partial \log \forward(x, \theta)}{\partial \psi}^{\transpose}} = \FisherInfo{\searchDist{\psi}(\theta)}.
\end{align}
Therefore, we have two options for the distance measure of probability distributions and the corresponding Fisher metric, and we should choose the one that is best suited to the problem under consideration.

\subsection{Approximation of Phenotype Gradient Estimator via Reparameterization Trick}
\label{sec:reparameterization-trick}
Most OIU problems may not fall within the case of Section~\ref{sec:analytical-form-of-full-gradient-estimator}.
To cover a wider range of problems, approximation methods to calculate the phenotype gradient estimator are crucial.
We derive it by the application of the reparameterization trick.

Let us assume that $\theta$ is decided deterministically from some noise: $\theta = \thetaFunc{\psi}{\noise}$, where $\noise$ follows a distribution $\noiseDist(\noise)$ and $\thetaFunc{\psi}{\cdot}$ is a deterministic function.
A typical example is Gaussian input noise: $\noise \sim \normal{\noise}{\bm{0}}{\identityMatrix}$ and $\thetaFunc{\psi}{\noise} = \bar \theta + \sigma \noise$, where $\psi = (\bar \theta, \sigma)$.
In this setting, the expectation $\smoothedObjFunc_W$ of the objective function becomes
\begin{align}
    \smoothedObjFunc(\psi) = \E{\noiseDist(\noise)}{
        \E{\sampleDist{\thetaFunc{\psi}{\noise}}(x)}{
            W(f(x))
        }
    }.
\end{align}
Then, we can calculate the gradient as
\begin{align}
    \frac{\partial \smoothedObjFunc_{W}}{\partial \psi} = \E{\noiseDist(\noise)}{ \frac{\partial \thetaFunc{\psi}{\noise}}{\partial \psi} \left[ \frac{\partial}{\partial \theta} \E{\sampleDist{\theta}(x)}{W(f(x))} \right]_{\theta = \thetaFunc{\psi}{\noise}} }.
\end{align}
By applying the log-derivative trick to the inner expectation, we have
\begin{align}
    \frac{\partial \smoothedObjFunc_{W}}{\partial \psi} = \E{\noiseDist(\noise)}{ \frac{\partial \thetaFunc{\psi}{\noise}}{\partial \psi} \E{\sampleDist{\theta}(x)}{W(f(x)) \left[ \frac{\partial \log \sampleDist{\theta}(x)}{\partial \theta} \right]_{\theta = \thetaFunc{\psi}{\noise}} } }.
\end{align}
By approximating the right hand side by samples, we have the following gradient estimator:
\begin{align}
    \label{eq:reparameterized-gradient-estimator}
    \reparameterizedEstimator \smoothedObjFunc_{W} = \frac{1}{\popSize} \sum_{i=1}^{\popSize} \frac{\partial \thetaFunc{\psi}{\noise_i}}{\partial \psi}W(f(x_i)) \left[ \frac{\partial \log \sampleDist{\theta}(x_i)}{\partial \theta} \right]_{\theta = \thetaFunc{\psi}{\noise_i}} .
\end{align}
This estimator is not Rao-Blackwellized since it contains irrelevant random variables $\noise_i$ on the right hand side, which are not deterministically decided by $\psi$ and $x$.
However, the dependency on $\noise$ is very weak in some cases.
We will see an example in Section~\ref{subsec:example-rl-reparameterization}.

\section{Phenotype Gradient Estimators for Representative OIU Problems}
\label{sec:example-of-paes}
In this section, we illustrate how the phenotype gradient estimator is obtained for specific OIU problems. 
For COIU which includes mixture of experts and classification with bandit feedback, see Appendix \ref{sec:COIU}.

\subsection{OIU with Gaussian Distributions and its Phenotype Gradient Estimator}
\label{subsec:gaussian-search-and-sampling}
Let us consider the case where $\searchDist{\psi}$ and $\sampleDist{\theta}$ are Gaussian as in Example~\ref{subsubsec:simple-example-of-continuous-optimization-with-input-noise}.
Concretely, $x, \theta \in \R^n$ and
\begin{align}
    \sampleDist{\theta}(x) = \normal{x}{\theta}{\sigma^2 \identityMatrix},
\end{align}
for some given $\sigma > 0$.
We also assume that $\psi = (\bar \theta, \Sigma)$ and
\begin{align}
    \searchDist{\psi}(\theta) = \normal{\theta}{\bar \theta}{\Sigma},
\end{align}
where $\Sigma$ is a covariance matrix.
Then, 
\begin{align}
    \forward(x) = \normal{x}{\bar \theta}{\Sigma + \sigma^2 \identityMatrix},
\end{align}
and we can calculate the phenotype gradient estimator analytically:
\begin{align}
    \label{eq:natural-full-gradient-estimator-for-gaussian-search-and-sampling}
    &\FisherInfo{\forward(x)}^{-1} \fullGradientGeneralArg{\bar \theta} \smoothedObjFunc_{W} = \frac{1}{\popSize}\sum_{i=1}^{\popSize} W(f(x_i))(x_i - \bar \theta),\\
    &\FisherInfo{\forward(x)}^{-1} \fullGradientGeneralArg{\Sigma} \smoothedObjFunc_{W} = \frac{1}{\popSize} \sum_{i=1}^{\popSize} W(f(x_i))  (x_i - \bar \theta)(x_i - \bar \theta)^{\transpose} - \Sigma - \sigma^2 \identityMatrix,
\end{align}
where the Fisher information matrix $\FisherInfo{\forward(x)}$ is defined by~\eqref{eq:fisher-information-matrix-search-dist}.
The derivation is given in Appendix ~\ref{subsec:gaussian-search-and-sampling-proof}.
This result is similar to the rank-$\mu$ update of CMA-ES~\cite{akimoto2010bidirectional, hansen2016cma,ollivier2017information}.
When we use the Fisher information matrix for $\searchDist{\psi}(\theta)$, we do not have a clear expression as in~\eqref{eq:natural-full-gradient-estimator-for-gaussian-search-and-sampling}. Therefore, in the Gaussian OIU case, it is more practical to use the natural phenotype gradient estimator for $\forward(x)$ rather than for $\searchDist{\psi}(\theta)$.

\subsection{Reinforcement Learning and its Phenotype Gradient Estimator}
\label{subsec:example-reinforcement-learning}
We next calculate the natural phenotype gradient estimator for RL problem introduced in Section~\ref{subsubsec:reinforcement-learning}.
We first consider a case where the state $\stateSp$ and the action $\actionSp$ are both discrete.
In addition, for simplicity, we fix the initial state $x_0$.
In this subsection, we take the genotype $\theta$ to be the policy $\pi$ itself.
Let $j(s', a, s) := \sum_{t=0}^T \mathbb{I}(s_t = s', a_t = a, s_{t+1} = s)$ be the empirical count of the transition $(s', a, s)$ in an observed sequence $x$.
Also, let $j(a, s) := \sum_{t=0}^{T-1} \mathbb{I}(s_t = s, a_t = a)$ be the empirical count of the state-action pair $(a, s)$ in the observed sequence $x$.
Then,
\begin{align}
    \sampleDist{\theta}(x) = \left(\prod_{t=0}^{T-1} \transitMatrix(s_{t+1} \mid s_t, a_t)\right) \prod_{s \in \stateSp, a \in \actionSp} \pi(a \mid s)^{j(a, s)}.
\end{align}
We introduce a Dirichlet distribution as the conjugate prior $\searchDist{\psi}$:
\begin{align}
    &\psi = (\alpha_{a,s})_{a \in \actionSp, s \in \stateSp},\\
    &\searchDist{\psi}(\theta) = \prod_{s \in \stateSp} \mathrm{Dir}(\bm{\pi}(\cdot \mid s) \mid (\alpha_{a,s})_{a \in \actionSp}).
\end{align}

Let $j_i$ be the empirical count of the transition in the $i$-th trial.
The phenotype gradient estimator is given by
\begin{align}
    \label{eq:full-gradient-estimator-for-reinforcement-learning-with-discrete-state-and-action}
    \fullGradientGeneralArg{\alpha_{a,s}} \smoothedObjFunc_{W} = \frac{1}{\popSize}\sum_{i=1}^{\popSize}W(f(x_i)) \left[ \left( \sum_{l=0}^{j_i(a,s)-1}  \frac{1}{l + \alpha_{a,s}}\right) -  \left( \sum_{l=0}^{\sum_{a' \in \actionSp} j_i(a',s) -1} \frac{1}{l + \sum_{a' \in \actionSp} \alpha_{a',s}}\right) \right].
\end{align}
The derivation is given in Section~\ref{subsec:example-reinforcement-learning-proof}.

\subsection{Application of Reparameterization Trick to RL}
\label{subsec:example-rl-reparameterization}

When the states and actions are continuous, we need to utilize the reparameterization trick to approximate the phenotype gradient estimator. In contrast to Section~\ref{subsec:example-reinforcement-learning}, here we take the genotype $\theta$ to be logit of the policy:
\begin{align}
    \label{eq:rl-genotype-logit}
    \theta &:= (\theta_{a,s})_{a \in \actionSp, s \in \stateSp},\\
    \qquad
    \pi_{\theta}(a \mid s) &= \frac{\exp(\theta_{a,s})}{\sum_{a' \in \actionSp} \exp(\theta_{a',s})}.
\end{align}
We assume the ZOO setting where the population is Gaussian.
Concretely, we assume
\begin{align}
    \theta = \thetaFunc{\psi}{\noise} := \bar \theta + \sigma \noise,\\
    \noise \sim \normal{\noise}{\bm{0}}{\identityMatrix},
\end{align}
where $\sigma > 0$ is a fixed hyperparameter and $\bar \theta$ is the parameter to be optimized.
We identify $\bar \theta$ with $\psi$ in the following discussion.

Let us compute the reparameterized gradient estimator~\eqref{eq:reparameterized-gradient-estimator}.
We decompose the empirical count $j(a, s) := \sum_{t=0}^{T-1} \mathbb{I}(s_t = s, a_t = a)$ as $j(a, s) = j(s) j(a \mid s)$, where $j(s) = \sum_{a \in \actionSp} j(a, s)$ is the empirical count of the state $s$ and $j(a \mid s) = j(a, s) / j(s)$ is the empirical conditional distribution of the action $a$ given the state $s$.
Then, the reparameterized gradient estimator~\eqref{eq:reparameterized-gradient-estimator} is given by
\begin{align}
    \label{eq:reparameterized-gradient-estimator-example}
    \reparameterizedEstimator_{a,s} \smoothedObjFunc_{W} = \frac{1}{\popSize} \sum_{i=1}^{\popSize} W(f(x_i)) 
    \left[ j(a, s) - j(s) \pi_{\thetaFunc{\psi}{\noise}}(a \mid s) \right].
\end{align}
The derivation is given in Section~\ref{subsubsec:example-rl-reparameterization-proof}.
This update rule resembles ancestral reinforcement learning~\cite{nakashima2024ancestralreinforcementlearningunifying} in that a descendant generation tends to repeat the action of its ancestors with higher cumulative rewards.

This estimator is not an exact Rao-Blackwellization, since it contains $\noise$.
Yet, the dependency on $\noise$ is very weak when $\sigma$ is small.
This fact is verified by the Taylor expansion of $\pi_{\thetaFunc{\psi}{\noise}}(a \mid s)$ around $\theta = \thetaFunc{\psi}{\noise}$:
\begin{align}
    \pi_{\thetaFunc{\psi}{\noise}}(a \mid s)  
    &= \pi_{\bar \theta}(a \mid s) + \left[ \mathrm{diag}(\bm{\pi}_{\bar \theta}(\cdot \mid s)) - \bm{\pi}_{\bar \theta}(\cdot  \mid s) \bm{\pi}_{\bar \theta}(\cdot \mid s)^{\transpose} \right] \sigma \noise + O(\|\sigma \noise\|^2)\\
    &= \pi_{\bar \theta}(a \mid s) + O(\|\sigma \noise\|).
\end{align}
Substituting this expansion into~\eqref{eq:reparameterized-gradient-estimator-example}, we obtain
\begin{align}
    \reparameterizedEstimator_{a,s} \smoothedObjFunc_{W} &= \frac{1}{\popSize} \sum_{i=1}^{\popSize} W(f(x_i)) \left[ j(a, s) - j(s) \left( \pi_{\bar \theta}(a \mid s) + O(\|\sigma \noise\|) \right) \right] \notag \\
    &= \underbrace{\frac{1}{\popSize} \sum_{i=1}^{\popSize} W(f(x_i)) \left[ j(a, s) - j(s) \pi_{\bar \theta}(a \mid s) \right]}_{\text{independent of } \noise} + O(\|\sigma \noise\|).
\end{align}
The leading term is independent of $\noise$ and thus Rao-Blackwellized, while the remaining term is only $O(\|\sigma \noise\|)$.

\section{Numerical Experiments}
\label{sec:numerical-experiments}
In this section, we demonstrate the superiority of the phenotype gradient estimator and associated PAES to the conventional genotype gradient estimator and the associated ES by using three classes of problems.
We observe that the phenotype gradient estimator consistently outperforms the genotype estimator across all classes tested.

\subsection{Gaussian OIU}
\label{sec:gaussian-nio}
We first evaluate the performance of ES and PAES on the OIU with the Gaussian distribution problem introduced in Section~\ref{subsec:gaussian-search-and-sampling}.

\subsubsection{Setup}
The objective function $f$ is chosen from the BBOB test suite~\cite{hansen2021coco}.
We note that BBOB benchmarks are minimization problems and, in this section, we use PAES for minimization problems by transforming $f$ to $-f$.
The search space is $d=40$ dimensional, and the population size is $100$. We run each algorithm for up to $5000$ generations and repeat the experiment with 10 different random seeds.
The search space is $[-5, 5]^d$ for all functions.
The magnitude $\sigma$ of the input noise is set to $0.1$.
We initialize the population as $\normal{\theta}{0}{\frac{1}{2}\identityMatrix
  }$.
The population is represented as a Gaussian distribution $\normal{\cdot}{\bar \theta}{\Sigma}$, where both mean $\bar \theta$ and covariance matrix $\Sigma$ are updated.
We utilized the vanilla SGD with learning rate $0.01$ for both ES and PAES.
For the weighting function $W$, we use a rank-based linear function commonly used in CMA-ES~\cite{hansen2016cma}.
Let $\mathrm{rank}(x_i) \in \{1, \ldots, N\}$ be the rank of $x_i$ when the population is sorted by $F(x_i)$ in descending order.
Let $p \in (0,1]$ be a hyperparameter that specifies the fraction of the population to receive a nonzero weight, and define $\xi := \lfloor p N \rfloor$ as the number of selected individuals.
We set $p = 0.8$ (hence $\xi = 80$ for $N=100$) and define
\begin{align}
  W(F(x_i)) =
  \begin{cases}
    w_{\max} \cdot \dfrac{\xi - \mathrm{rank}(x_i)}{\xi- 1} & \mathrm{rank}(x_i) \leq \xi, \\
    0                                                        & \mathrm{rank}(x_i) > \xi,
  \end{cases}
\end{align}
where $w_{\max} = 0.1$.
For ES, we estimate the gradient via rank-$\mu$ CMA-ES~\cite{hansen2016cma} as
\begin{align}
  &\frac{\partial}{\partial \bar \theta} \smoothedObjFunc_W = \frac{1}{N} \sum_{i=1}^N W(F(x_i)) \left(\theta_i - \bar \theta\right),\\
  &\frac{\partial}{\partial \Sigma} \smoothedObjFunc_W = \frac{1}{N} \sum_{i=1}^N W(F(x_i)) \left((\theta_i - \bar \theta)(\theta_i - \bar \theta)^\top - \Sigma\right).
\end{align}
For PAES, we use the update rule~\eqref{eq:natural-full-gradient-estimator-for-gaussian-search-and-sampling}.

\subsubsection{Performance evaluation}
Figure~\ref{fig:bbob-all-functions} shows the mean difference between the final objective values and the optimum values for 10 runs of all 24 BBOB functions.
The figure indicates that PAES consistently achieves lower errors than ES.
Figure~\ref{fig:bbob-selected} shows the convergence curves for three representative cases.
A typical case is the sphere function (f1) where PAES converges faster than ES and achieves a lower objective value at the end of the optimization.
Another case is Schwefel x*sin(x) (f20), where PAES finds a significantly better solution than ES (see also Figure~\ref{fig:bbob-all-functions}).
This function has many local minima on a scale comparable to the input noise, and thus the genotypic information $F(\theta_i)$ is obscured by the noise.
As a result, ES was unable to find the optimal solution, while PAES found the optimal solution using phenotypic information $x_i$ directly.
The last case is the Katsuura function (f23), where both methods failed to minimize the objective function and plateaued at a similar level.
The function (f23) has many local minima arranged in a lattice pattern within the input noise scale.
The input noise averages such local structures and the objective function $F(\theta)$ becomes flat, making it difficult for both ES and PAES to minimize the objective function.

The convergence curves for all $24$ functions are provided in Appendix~\ref{sec:appendix-bbob}.
We also note that the computation times for 5000 generations for ES and PAES on all $24$ BBOB functions are comparable (Appendix~\ref{sec:appendix-bbob}).

\begin{figure}[t]
  \centering
  \includegraphics[width=\textwidth]{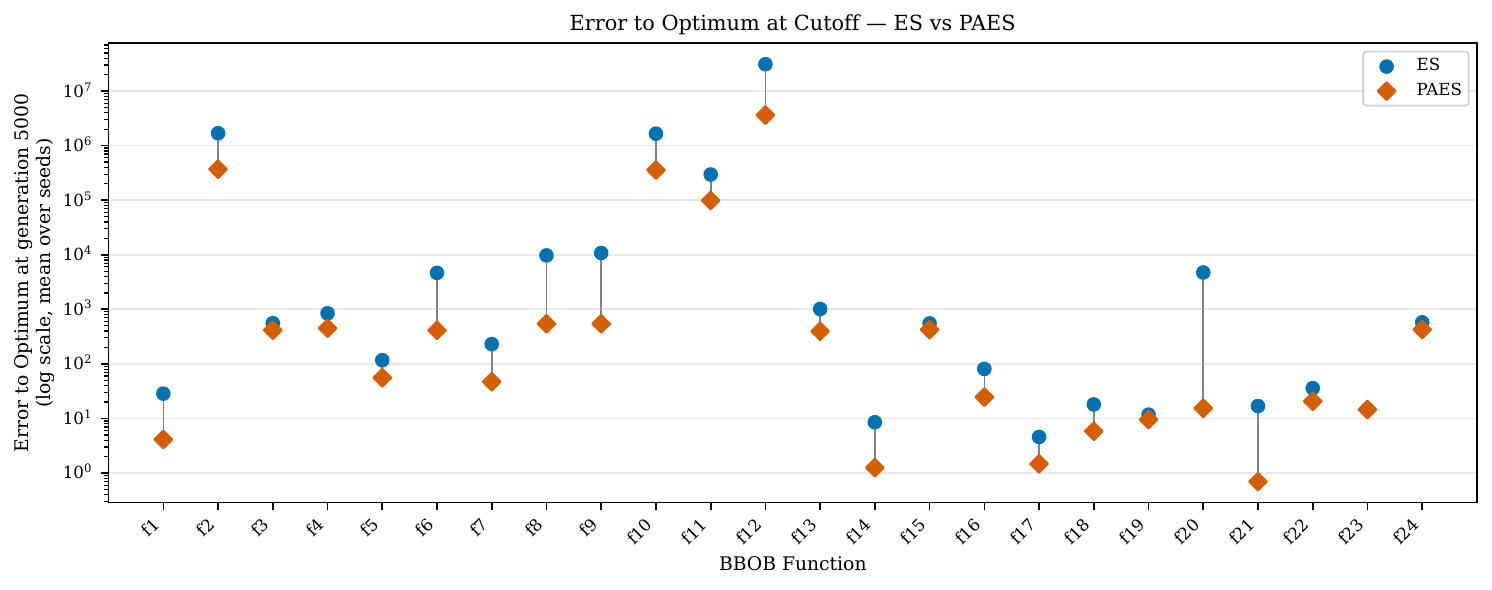}
  \caption{The difference of the final objective values to the optimum values on all 24 BBOB functions ($d=40$, population size $100$,
    $5000$ generations, $10$ seeds). The difference is averaged over 10 different random seeds.
    PAES (orange) outperforms ES (blue) on most functions.}
  \label{fig:bbob-all-functions}
\end{figure}

\begin{figure}[t]
  \centering
  \begin{minipage}[t]{0.32\textwidth}
    \centering
    \includegraphics[width=\textwidth]{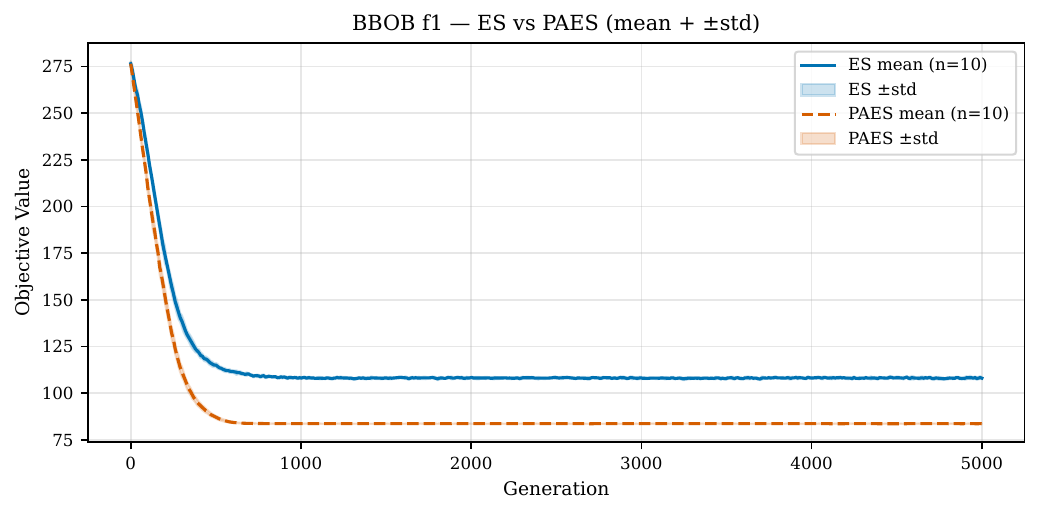}
    \begin{center}\small (a) f1 (Sphere): typical case.\end{center}
  \end{minipage}\hfill
  \begin{minipage}[t]{0.32\textwidth}
    \centering
    \includegraphics[width=\textwidth]{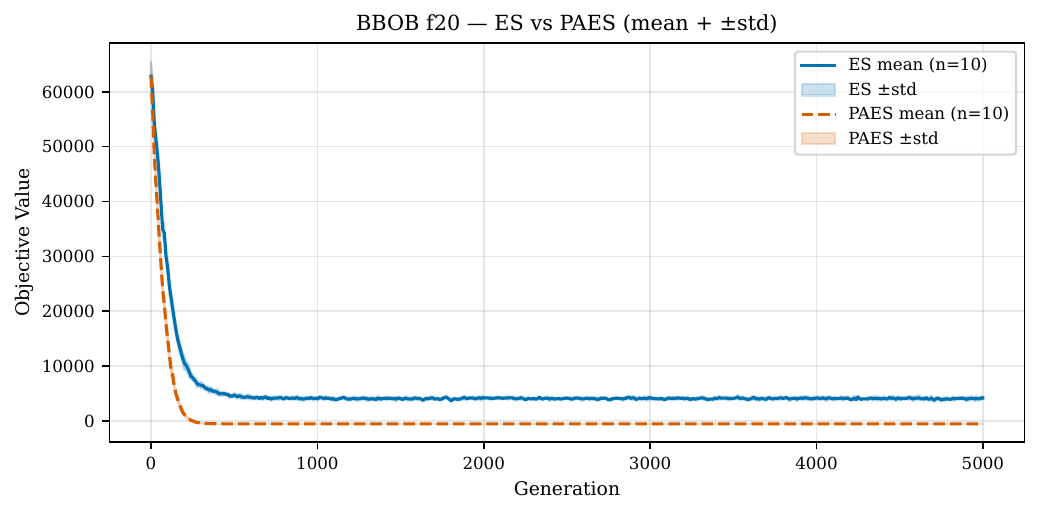}
    \begin{center}\small (b) f20 (Schwefel x*sin(x)): PAES outperforms ES.\end{center}
  \end{minipage}\hfill
  \begin{minipage}[t]{0.32\textwidth}
    \centering
    \includegraphics[width=\textwidth]{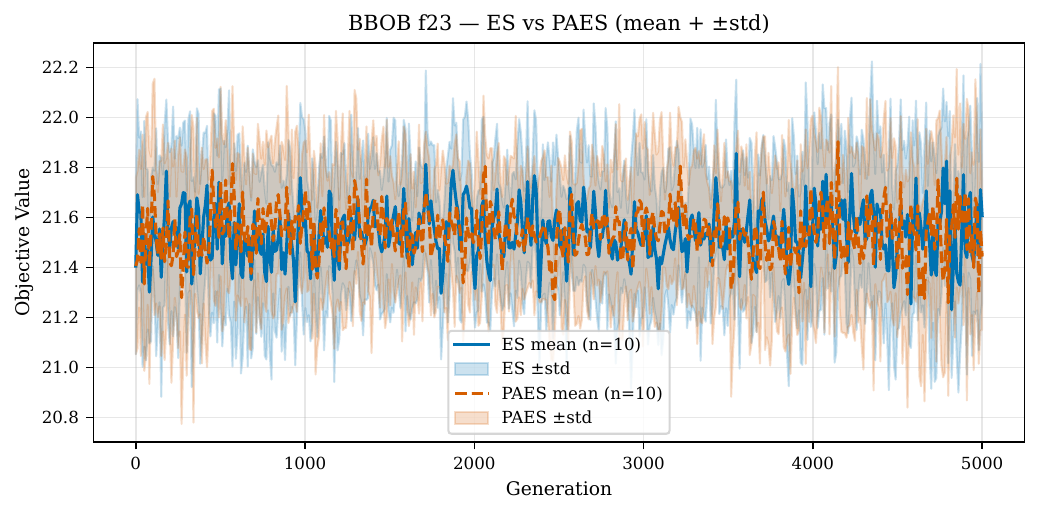}
    \begin{center}\small (c) f23 (Katsuura): both methods fail to minimize.\end{center}
  \end{minipage}
  \caption{Selected convergence curves (mean $\pm$ std, 10 seeds).
    PAES (orange) vs.\ ES (blue).}
  \label{fig:bbob-selected}
\end{figure}

\subsection{Reinforcement Learning (Exact Rao-Blackwellization)}
\label{sec:rl-exact-rao-blackwell}

Next, we evaluate the performance of PAES for an RL problem with discrete states and actions where the exact phenotype gradient estimator (~\eqref{eq:full-gradient-estimator-for-reinforcement-learning-with-discrete-state-and-action}) is used. 
\subsubsection{Setup}
We adopt the FrozenLake-v1 environment~\cite{towers2024gymnasium}.
The environment consists of $16$ states and $4$ actions (left/down/right/up) in a $4 \times 4$ grid world.
For ES, we parameterize the policy using a Q-table, each of which is initialized by a normal distribution $\mathcal{N}(0, 0.01)$.
For PAES, we parametrize the policy as a set of Dirichlet distributions, each of which corresponds to a state $x$ and whose parameter is initialized by $\alpha_s = (1.0, 1.0, 1.0, 1.0)$.
The population size is set to $100$, and we optimize the policy using vanilla SGD with a learning rate $0.01$ for both ES and PAES.
The weighting function $W$ is defined by using the cumulative reward $r_i$ of each policy in the population as follows:
\begin{align}
  W(r_i) = \frac{r_i - \bar{r}}{\hat{\sigma}_r},
  \label{eq:rl-frozenlake-weight}
\end{align}
where $\bar{r} = \frac{1}{N}\sum_{j=1}^N r_j$ and $\hat{\sigma}_r = \sqrt{\frac{1}{N}\sum_{j=1}^N (r_j - \bar{r})^2}$.

\subsubsection{Performance evaluation}

Since the performance of ES and PAES may depend on the hyperparameters $\sigma$ and $S = \sum_{a \in \actionSp} \alpha_{a,s}$ respectively,  we ran each algorithm with different random seeds and hyperparameters.
We measure the performance of the algorithm at each step by the cumulative reward of the updated policy, which is obtained by normalizing the parameters $\alpha_s$ of the Dirichlet distribution for PAES and directly from the Q-table for ES.
This evaluation of the reward is different from the evaluation of the cumulative rewards in the population in that we make the policy deterministic by taking the action with highest probability for each state.
Figure~\ref{fig:frozenlake-main} (a) shows the learning curves of each algorithm for the best hyperparameter setting, where the vertical axis shows the cumulative reward of the updated policy.
Here, we choose the best hyperparameter setting for each algorithm by the speed of convergence of the cumulative reward and stability of the learning curve (see Appendix~\ref{sec:appendix-rl-frozenlake-sweep} for details).
Although we run the experiment until $500$ generations, we show only the first $250$ generations for better visibility.
PAES converges to a high cumulative reward of the updated policy faster than ES.
Figure~\ref{fig:frozenlake-main} (b,c) shows the sensitivity of the cumulative reward of the updated policy at the last $50$ generations to the hyperparameters.
We observe that the performance of ES is better when the $\sigma$ is no less than $0.1$, which means that the variance of the population is large.
On the other hand, the performance of PAES is better when $S$ is smaller than $1$, which means that the variance of the population is large like ES.
The comparison of computation time is provided in Appendix~\ref{sec:appendix-rl}, where PAES converges faster than ES in terms of wall-clock time as well.

\begin{figure}[t]
  \centering
  \begin{minipage}[t]{0.32\textwidth}
    \centering
    \includegraphics[width=\textwidth]{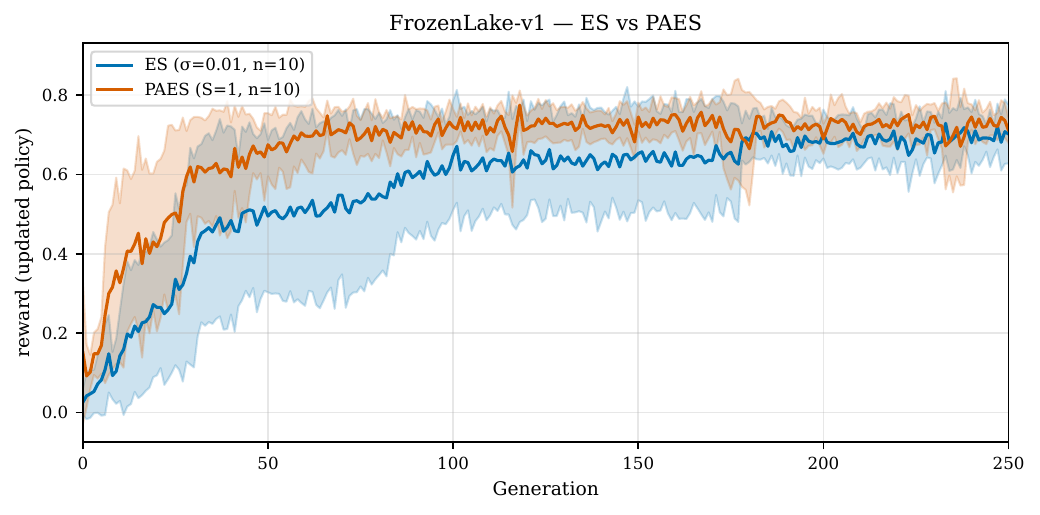}
    \begin{center}\small (a) Learning curves at the best setting.\end{center}
  \end{minipage}\hfill
  \begin{minipage}[t]{0.32\textwidth}
    \centering
    \includegraphics[width=\textwidth]{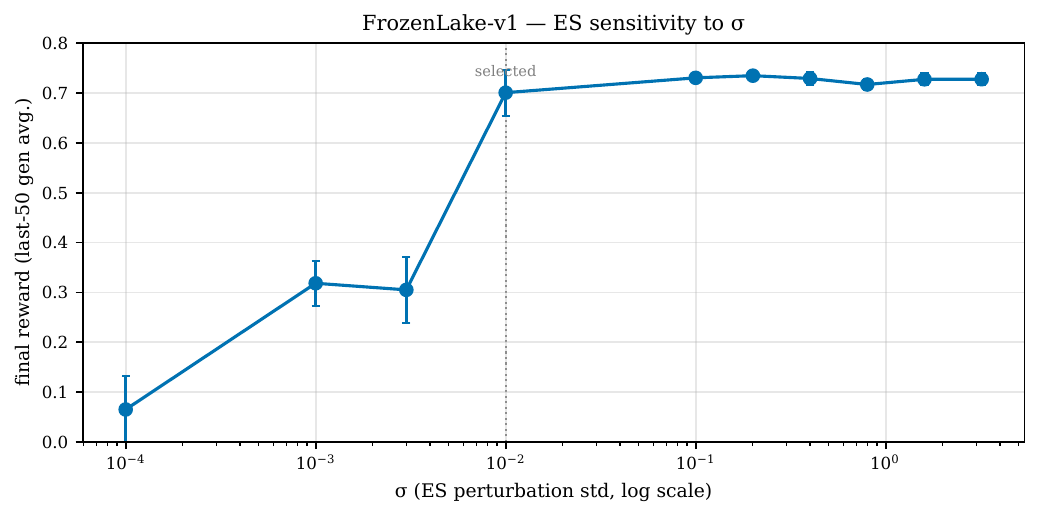}
    \begin{center}\small (b) ES: sensitivity to $\sigma$.\end{center}
  \end{minipage}\hfill
  \begin{minipage}[t]{0.32\textwidth}
    \centering
    \includegraphics[width=\textwidth]{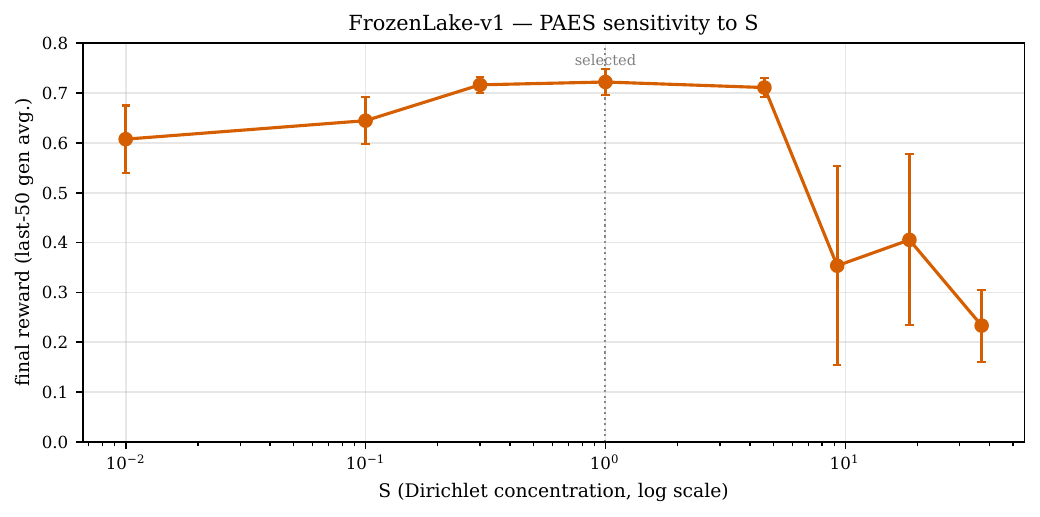}
    \begin{center}\small (c) PAES: sensitivity to $S$.\end{center}
  \end{minipage}
  \caption{(a) Learning curves of the updated policy for ES ($\sigma=0.01$) and PAES ($S=1$)
    on FrozenLake-v1, each averaged over 10 seeds (mean $\pm$ std).
    Both algorithms reach a comparable final reward, but PAES converges substantially faster.
    Although we run experiment until 500 generations, we show only the first 250 generations for better visibility.
    (b, c) The sensitivity of the final reward of the updated policy to hyperparameters $\sigma$ and $S$ for ES and PAES, respectively,
    We show the average of the last 50 generations (450-500 generation) as line plot with standard deviation as error bars.
    We observe that both algorithms have the saturation of reward when the variance of the population is large, i.e., $\sigma$ is no less than $0.1$ for ES and $S$ is smaller than $1$ for PAES.
    We choose $\sigma=0.01$ for ES and $S=1$ for PAES as the best hyperparameter setting for each algorithm by considering the final reward, convergence speed, and stability of the learning curve.
    }
  \label{fig:frozenlake-main}
\end{figure}

\subsection{Reinforcement Learning (Reparameterization Trick Approximation)}
\label{sec:rl-reparameterization}
We evaluate the performance of PAES for more complex RL tasks by applying the reparameterization trick approximation of the phenotype gradient estimator~\eqref{eq:reparameterized-gradient-estimator-example}.
In the following, we present learning curves in terms of generations.
The comparison of computation time is provided in Appendix~\ref{sec:appendix-rl}, where the learning curves are similar to those in the main text.

\subsubsection{CartPole}
We apply the approximated PAES with reparameterization trick to the CartPole-v1 environment~\cite{towers2024gymnasium}.
We note that the PAES with the exact Rao-Blackwellization ~\eqref{eq:full-gradient-estimator-for-reinforcement-learning-with-discrete-state-and-action} is not applicable to this environment because the state space is continuous.
The maximum step for each episode is set to $500$, which is equal to the maximum reward for each episode.
We introduce early stopping when the mean reward of the population reaches $475$.
The policy is represented by a three-layer neural network with $32$ hidden units with the $tanh$ activation function.
The final layer outputs the action probabilities by applying the softmax function.
The policy is initialized by the Xavier initialization~\cite{glorot2010understanding}.
We add Gaussian noise with standard deviation $\sigma = 0.1$ to the policy parameters.
The population size is set to $100$.
We optimize the policy by Adam with learning rate $0.001$, $\beta_1 = 0.9$, $\beta_2 = 0.999$ for both ES and PAES.
We define the weighting function $W$ by using the cumulative reward $r_i$ of each policy in the population as follows:
\begin{align}
  W(r_i) = \frac{\mathrm{rank}(r_i)}{N - 1} - \frac{1}{2},
  \label{eq:rl-cartpole-weight}
\end{align}
where $\mathrm{rank}(r_i) \in \{0, 1, \ldots, N{-}1\}$ is the zero-indexed rank of $r_i$ among $\{r_1, \ldots, r_N\}$.
This weighting function is commonly used in the application of ES to RL problems~\cite{salimans2017evolution}.
For ES, we use antithetic sampling, which is a common variance reduction technique for ES~\cite{salimans2017evolution}.
To apply the update rule of PAES~\eqref{eq:reparameterized-gradient-estimator-example}, we collect the empirical count $j(a,s)$ of each policy in the population and then compute the gradient of the original policy by back-propagating the gradient ~\eqref{eq:reparameterized-gradient-estimator-example} to the network.

Figure~\ref{fig:cartpole} shows the learning curves of ES and PAES for CartPole-v1.
PAES converges significantly faster than ES in terms of the number of generations, which demonstrates the effectiveness of PAES with the reparameterization trick in RL problems.

\subsubsection{Pendulum}
Next, we evaluate the performance of PAES for the Pendulum-v1 environment~\cite{towers2024gymnasium}, where both state and action spaces are continuous.
The policy is represented by a three-layer MLP with $64$ hidden units with the $tanh$ activation function, which directly outputs the continuous action.
When we generate a population in both ES and PAES, the policy parameters $\theta$ are perturbed by Gaussian noise with standard deviation $\sigma$.
Here, the value of $\sigma$ starts from $0.05$ and linearly decays to $0.01$ over $3000$ generations.
In addition, for PAES, we add Gaussian noise with standard deviation $\sigma$ to the action output of the policy.
This noise corresponds to the phenotypic noise in PAES and accelerates exploration in the action space.
For ES, we found that adding noise to the action worsens the performance (results not shown), and thus we show the result without adding noise to actions.
We set the threshold for early stopping to $-200$.
The other settings are the same as those in CartPole-v1.

Figure~\ref{fig:cartpole} (b) shows the learning curves of ES and PAES.
The vertical axis shows the deterministic reward of the updated policy, which is evaluated by not adding the action noise to the policy.
The result shows that PAES converges faster than ES in terms of generations.

\subsubsection{Swimmer}
Finally, we evaluate the performance of PAES in the Swimmer-v5 environment~\cite{towers2024gymnasium}, which is a more complex continuous control environment.
The setting is the same as those in Pendulum-v1, except for the following points.
The standard deviation $\sigma$ of the noise added to policies $theta$ and actions starts from $0.2$ and linearly decays to $0.05$ over $5000$ generations.
We added larger noise because, with smaller noise starting from $\sigma = 0.05$, both ES and PAES failed to find a successful policy within $5000$ generations for about half of the seeds.
This is reasonable because the Swimmer environment is known to have many local minima, and thus exploration is necessary to find the successful policy.
We balance exploration and stability after convergence by tuning the standard deviation $\sigma$ of the noise.

Figure~\ref{fig:cartpole} (c) shows the learning curves of ES and PAES for Swimmer-v5.
PAES converges faster than ES in terms of generations, demonstrating that phenotypic information can also accelerate learning in complex continuous control environments.

\begin{figure}[t]
  \centering
  \begin{minipage}[t]{0.32\textwidth}
    \centering
    \includegraphics[width=\textwidth]{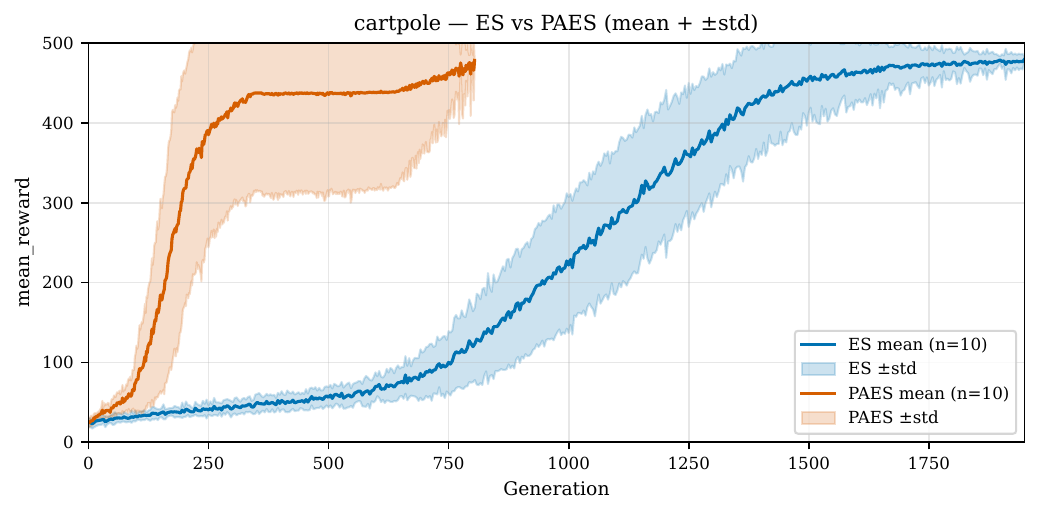}
    \begin{center}\small (a) CartPole-v1.\end{center}
  \end{minipage}\hfill
  \begin{minipage}[t]{0.32\textwidth}
    \centering
    \includegraphics[width=\textwidth]{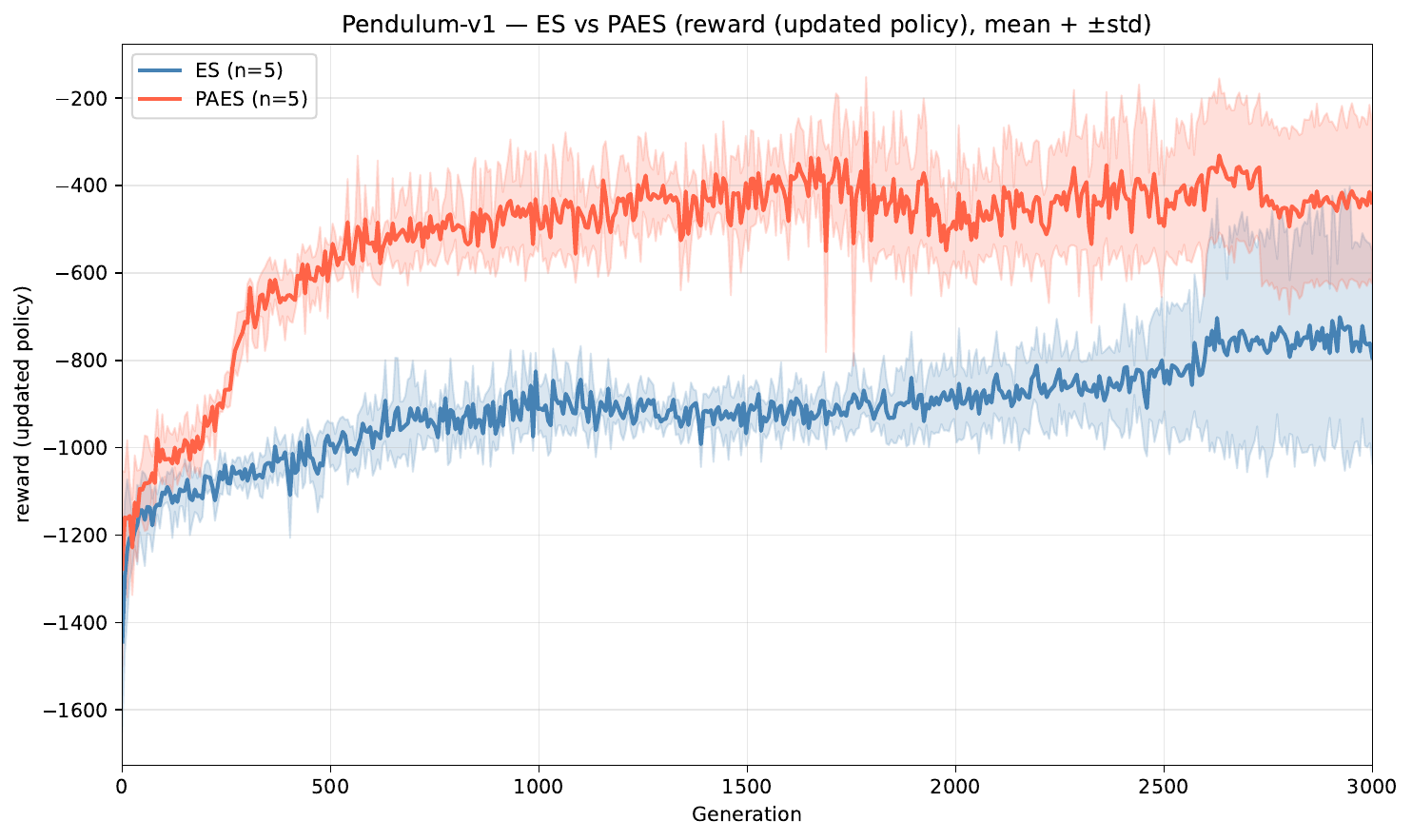}
    \begin{center}\small (b) Pendulum-v1.\end{center}
  \end{minipage}\hfill
  \begin{minipage}[t]{0.32\textwidth}
    \centering
    \includegraphics[width=\textwidth]{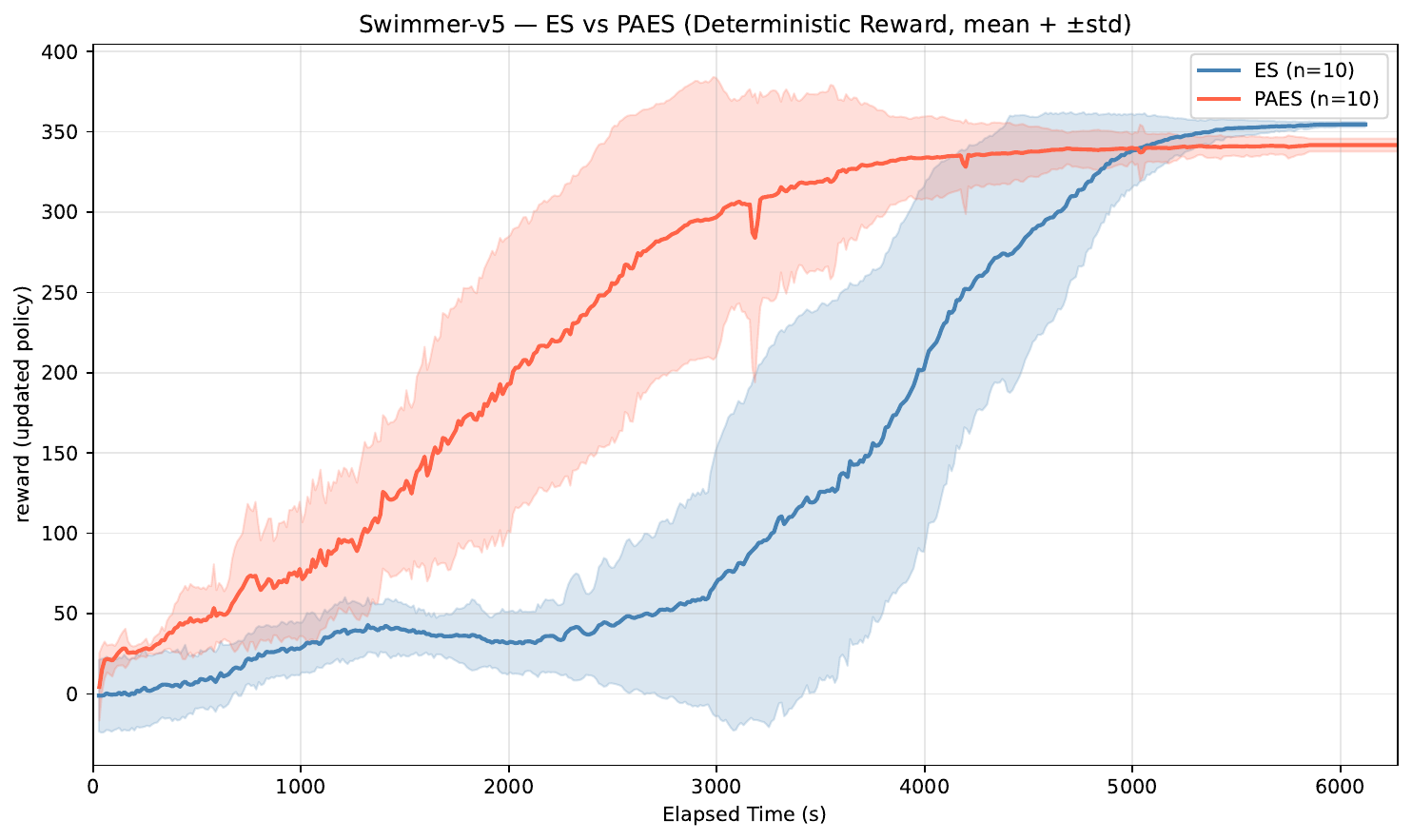}
    \begin{center}\small (c) Swimmer-v5.\end{center}
  \end{minipage}
  \caption{Learning curves of ES and PAES for continuous control environments
    (mean $\pm$ std over multiple seeds).
    PAES (orange) vs.\ ES (blue).
    See Appendix~\ref{sec:appendix-rl} for computation time comparison.}
  \label{fig:cartpole}
\end{figure}

\section{Conclusion}
\label{sec:conclusion}

In this paper, we have investigated the question of whether the realization of the input is useful to accelerate the ES in OIU problems.
We showed that the acceleration of the ES is achieved by PAES in which the phenotype gradient estimator utilizes the information of input realization, i.e., phenotype information. 
Then, we generalized this result by clarifying that the acceleration is achieved by Rao-Blackwellization, which reduces the variance of the gradient estimation by utilizing the realization of the input.
We also provide both analytical forms and an approximation of the phenotype gradient estimator in PAES to cover a variety of problems.
Numerical experiments demonstrated the superiority of PAES, which is consistent with our theoretical results.
PAES achieves faster convergence than ES from simple optimization to more complex RL benchmarks.
We found that the computation time of PAES is comparable to that of ES, which means that PAES is more efficient than ES even in terms of wall-clock time.
Furthermore, to obtain finer gradient estimators, PAES could be combined with several other techniques for RL, such as the baseline control~\cite{greensmith2001variance}.

A limitation of PAES is that it requires the input distribution to be in some well-behaved class to calculate the phenotype gradient analytically.
To apply PAES to a broader class of problems, we need to develop an approximation method for the phenotype gradient estimator.
A possible approach is to use the phenotypic information to imitate the input $x$ with a higher objective value in the population.
Indeed, the phenotype gradient estimator for the Gaussian OIU problem~\eqref{eq:natural-full-gradient-estimator-for-gaussian-search-and-sampling} can be interpreted as imitating the input $x$ with a higher objective value in the population.
In addition, the reparameterization trick~\eqref{eq:reparameterized-gradient-estimator-example} employed for RL can be regarded as imitating the action of agents with a higher reward in the population.
These results indicate that imitation of the ancestors with a higher objective value in the population can accelerate the ES, even though it may not necessarily be represented as Rao-Blackwellization.

The application of OIU has been expanding from classical examples of manufacturing to robotics and RL.
Another application of OIU would be Large Language Models, in which the stochastic output is guided by optimizing prompts and context information fed to LLM.
This expansion indicates possible broad applications of this paper.
Moreover, there are several solutions to OIU: ES, gradient-based methods, and Bayesian optimization.
Although we focused only on ES and its extension PAES, a similar idea may be applicable to other solutions.
For such directions, the observation that Rao-Blackwellization via the realization of the input resembles imitation learning might be useful to construct improved algorithms.

\section{Acknowledgment}
This research is supported by Toyota Konpon Research Institute,
Inc., JST CREST (JPMJCR25Q2) and
JSPS Grant-in-Aid for Transformative Research Areas (A) (25H01365).

\bibliography{main.bib}

\appendix

\section{Related Work to OIU and PAES}
\label{sec:related-work}
ES~\cite{rechenberg1973evolutionary} is applied to a broad class of problems, especially when the gradient of the objective function is not available.
ES represents the population as a probability distribution and aims to maximize the expected value of the objective function over population.
The expected value is maximized via the stochastic gradient ascent, where the gradient is estimated by using the samples of objective values in the population.
Variants of ES, such as CMA-ES~\cite{hansen2016cma,hensen2001completely}, Natural Evolution Strategy (NES)~\cite{wierstra2014natural}, and Information Geometric Optimization (IGO)~\cite{ollivier2017information} are widely applied to various problems and have high performance on benchmark functions~\cite{hansen2010comparing, rios2013derivativefree}.
In the context of black-box optimization, ES is also called zeroth-order optimization (ZOO), and its properties such as convergence rate and sample complexity have been studied~\cite{nesterov2017random,flaxman2005online,shamir2017optimal,duchi2015optimal,ye2025onepoint}.
In this paper, we improve ES for OIU by tapping into the information of realized input $x$.

Variance reduction techniques for gradient estimation have been studied in various contexts of machine learning.
One technique is Rao-Blackwellization, in which we can reduce the variance of estimator by taking conditional expectation with respect to irrelevant variable.
Rao-Blackwellization is applied to variational inference~\cite{salimans2013fixed, ranganath2014black,lam2026raoblackwellised}, the estimation of the gradient of discrete variables~\cite{yin2019arsm, liu2019raoblackwellised}, and the construction of action-dependent baselines for policy gradient methods~\cite{wu2018variance}.
Another technique is the reparameterization trick (also known as stochastic backpropagation~\cite{rezende2014stochastic}), which, for a compositional objective function, computes some part of the objective function analytically and then estimates the gradient of the other part by sampling.
The reparameterization trick is empirically shown to have a smaller variance~\cite{schulman2015gradient, mohamed2020monte} and
theoretical analysis is provided for some idealized cases~\cite{xu2019variance}.
The reparameterization trick is also applied to the variational auto encoder~\cite{kingma2014auto}, and variational inference~\cite{lam2026raoblackwellised}.

\section{Contextual Optimization under Input Uncertainty (COIU)}
\label{sec:COIU}
As an extension of OIU, we also define a \emph{Contextual Optimization under Input Uncertainty} (COIU).
At each trial, we are given a context $c \in \ContextSp$ which is drawn from some distribution $\nu(c)$.
For each $\theta \in \SearchSp$, we define a conditional probability distribution $\sampleDist{\theta}(x \mid c)$ on $\SampleSp$.
For a function $f \colon   \SampleSp \times \ContextSp \to \R$, the objective function is given by
\begin{align}
    \label{eq:objective-func-with-context}
    \baseObjFunc(\theta) := \E{\nu(c)}{ \E{\sampleDist{\theta}(x \mid c)}{f(x,c)}}.
\end{align}

\subsection{Examples of the Problem}
\label{subsubsec:example-problem2}

\subsubsection{Mixture of Experts}
\label{subsubsec:mixture-of-experts}
Let us consider a classification problem.
We are given a context $c$ that follows a distribution $\nu(c)$, and we need to output the corresponding label in a discrete set $\SampleSp$.
The loss function is $f(c,x)$.
One solution to this classification problem is MoE.
Assume that we have a set of experts $\phi = \{\phi_i\}_{i=1}^k$, each of which defines a distribution $\pi_{\phi_i}(x \mid c)$ over the output labels.
In MoE, we first choose an expert $i$ with probability $\alpha(i \mid c, \rho)$, where $\rho$ is a parameter.
Then, the chosen expert outputs a label with probability $\pi_{\phi_i}(x \mid c)$.
Our aim is to optimize each experts $\phi$ and the choice rule $\rho$ of experts to minimize the expected loss.
The expected loss when we choose the $i$-th expert is given by
\begin{align}
    \baseObjFunc(i) = \E{\nu(c)}{ \E{\pi_{\theta_i}(x \mid c) }{f(c,x)}}.
\end{align}
Therefore,  the optimization of $\phi_i$ is formulated in the framework of~\eqref{eq:objective-func-with-context}.
The expected loss when we choose exert stochastically from $\alpha(i \mid c, \rho)$ is
\begin{align}
    \label{eq:expected-loss-moe-full}
    \baseObjFunc(\phi, \rho) = \E{\nu(c)}{ \E{\pi_{\theta_i}(x \mid c) \alpha(i \mid c, \rho)}{f(c,x)}}.
\end{align}
In Section~\ref{subsec:moe-paes}, we will solve MoE by ES and see that $\alpha(i \mid c, \rho)$ corresponds to a population distribution of ES and the expected loss to $\smoothedObjFunc$ in~\eqref{eq:def-smoothe-obj}.

\subsubsection{Classification with Bandit Feedback}
\label{subsubsec:classification-with-bandit-feedback}
We consider a classification problem with bandit feedback~\cite{kakade2008efficient}.
The classification problem is the same setting as the previous subsection.
We note that this setting is different from the usual classification problems in that we do not know the correct class for $c$.
Suppose that we have a stochastic classification algorithm $\pi_{\theta}(x \mid c)$, which is a conditional probability distribution on $\SampleSp$ given $c$.
The expected loss is given by
\begin{align}
    \baseObjFunc(\theta) = \E{\nu(c)}{ \E{\pi_{\theta}(x \mid c)}{f(c, x)}}.
\end{align}
By setting
\begin{align}
    \sampleDist{\theta}(x) = \pi_{\theta}(x \mid c),
\end{align}
the expected value of the feedback is formulated in the framework of COIU~\eqref{eq:objective-func-with-context}.

\subsection{Phenotype gradient estimators for COIU}
We derive phenotype gradient estimators for COI in thie section for the two examples we introduced above.
\subsubsection{Mixture of Experts}
\label{subsec:moe-paes}
We consider the example in Section~\ref{subsubsec:mixture-of-experts}.
Let us denote the event that the $i$-th expert with parameter $\phi_i$ is chosen by $\theta_{i,\phi_i}$.
Also, let $\psi = (\phi, \rho)$.
We also define
\begin{align}
    \searchDist{\psi}(\theta_{i,\phi_i} \mid c) = \delta_{\phi_i}(\theta_{i,\phi_i}) \alpha(i \mid c, \rho),\\
    \sampleDist{\theta_{i, \phi_i}}(x \mid c) = \pi_{\phi_i}(x \mid c), 
\end{align}
where $\delta$ is the delta function.
Under this setting, the expected loss~\eqref{eq:expected-loss-moe-full} can be written as
\begin{align}
    F(\phi, \rho) = \E{\nu(c)}{
        \E{\searchDist{\psi}(\theta_{i,\phi_i})}{
            \E{\sampleDist{\theta_{i,\phi_i}(x \mid c)}}{f(x,c)}
        }
    } = \smoothedObjFunc(\psi).
\end{align}
Therefore, MoE is formulated as smoothed objective function of COIU (cf.~\eqref{eq:smoothe-obj-oiu}).

From this observation, we can calculate the phenotype gradient estimator for MoE.
We know that
\begin{align}
    \forward(x,  \theta_{i,\phi_i} \mid c, \psi) = \pi_{\phi_i}(x \mid c) \alpha(i \mid c, \rho),\\
    \forward(x \mid c, \psi) = \sum_{i=1}^k \pi_{\phi_i}(x \mid c) \alpha(i \mid c, \rho).
\end{align}
We also know that
\begin{align}
    \forward(\theta_{i,\phi_i} \mid c, x, \psi) = \frac{\pi_{\phi_i}(x \mid c) \alpha(i \mid c, \rho)}{\sum_{i=1}^k \pi_{\phi_i}(x \mid c) \alpha(i \mid c, \rho)}.
\end{align}
Therefore, the phenotype gradient estimator is given by
\begin{align}
    \fullGradientGeneralArg{\rho} \smoothedObjFunc_{W}
    &= \frac{1}{\popSize}\sum_{i=1}^{\popSize} W(f(x_i)) \frac{\partial \log \forward(x \mid c, \psi)}{\partial \rho}\\
    &= \frac{1}{\popSize}\sum_{i=1}^{\popSize} W(f(x_i)) \left[
        \frac{
            \pi_{\phi_i}(x \mid c) \frac{\partial}{\partial \rho} \alpha(i \mid c, \rho)
        }{
            \sum_{i=1}^k \pi_{\theta_i}(x \mid c) \alpha(i \mid c, \rho)
        }
    \right]\\
    &= \frac{1}{\popSize}\sum_{i=1}^{\popSize} W(f(x_i)) \forward(i \mid c, x, \rho) \frac{\partial}{\partial \rho} \log \alpha(i \mid c, \rho).
    \label{eq:moe-rho-paes}
\end{align}
Also, we have
\begin{align}
    \fullGradientGeneralArg{\phi_i} \smoothedObjFunc_{W}
    &= \frac{1}{\popSize}\sum_{i=1}^{\popSize} W(f(x_i)) \frac{\partial \log \forward(x \mid c, \psi)}{\partial \phi_i}\\
    &= \frac{1}{\popSize}\sum_{i=1}^{\popSize} W(f(x_i)) \left[ \frac{\alpha(i \mid c_i, \rho) \frac{\partial}{\partial \phi_i}  \pi_{\phi_i}(x_i \mid c_i)}{\sum_{i=1}^k \pi_{\theta_i}(x_i \mid c_i) \alpha(i \mid c_i, \rho)} \right]\\
    &=\frac{1}{\popSize}\sum_{i=1}^{\popSize} W(f(x_i)) \forward(i \mid c_i, x_i, \rho) \frac{\partial}{\partial \phi_i} \log \pi_{\phi_i}(x_i \mid c_i).
    \label{eq:moe-theta-paes}
\end{align}
These equations resembles the update rule of EM algorithm for MoE~\cite{jordan1993hierarchical} in that it utilizes a posterior distribution $\forward(i \mid c, x, \psi)$ to update the parameter $\psi$.

\subsubsection{Classification with Bandit Feedback}
\label{subsec:example-classification-with-bandit-feedback}
We next consider the example in Section~\ref{subsubsec:classification-with-bandit-feedback}.
We take the genotype $\theta$ to be the classification probability itself.
Also, we take a conjugate prior $\searchDist{\psi}$ to the multinomial distribution $\pi(\cdot \mid c)$ for each $c \in \ContextSp$ as
\begin{align}
    &\searchDist{\psi}(\theta \mid c) = \mathrm{Dir}(\pi(\cdot \mid c) \mid (\alpha_{x,c})_{x \in \SampleSp}).
\end{align}
Under this setting, the phenotype gradient estimator is given by
\begin{align}
    \label{eq:full-gradient-estimator-for-classification-with-bandit-feedback}
    \fullGradientGeneralArg{\alpha_{x,c}} \smoothedObjFunc_{W} = \frac{1}{\popSize}\sum_{i=1}^{\popSize}  \delta_{c, c_i} \delta_{x, x_i} W(f(x_i)) \left[ \frac{1}{\alpha_{x,c}} - \frac{1}{\sum_{x' \in \SampleSp} \alpha_{x',c}} \right].
\end{align}
The calculation is given in Section~\ref{subsec:example-classification-with-bandit-feedback-proof}.

\section{Proofs}
\subsection{Proofs in Section~\ref{sec:algorithm}}
\label{subsec:algorithm-proof}

\begin{proof}[Proof of Theroem~\ref{thm:full-gradient-estimator} and Theorem~\ref{thm:generalized-unbiasedness-of-rao-blackwellized-estimator}]
We prove Theorem~\ref{thm:generalized-unbiasedness-of-rao-blackwellized-estimator} which includes Theorem~\ref{thm:full-gradient-estimator} as a special case.
We first prove~\eqref{eq:generalized-unbiasedness-of-rao-blackwellized-estimator}.
From the definition of the Rao-Blackwellized estimator, we have
\begin{align}
    \E{\forward}{\RaoBlackwellize{\estimatorChar}{\{\suptime{x_i}{t}\}_{i \in [\popSize]}}}
    &= \E{
            \forward[\{\suptime{x_i}{t}\}_{i \in [\popSize]}]
        }{
            \E{\forward[\{ \suptime{\theta_i}{t} \}_{i \in [\popSize]} \mid \{\suptime{x_i}{t}\}_{i \in [\popSize]}]}{\estimator{\{\suptime{\theta_i}{t}, \suptime{x_i}{t}\}_{i \in [\popSize]}}}
        } \\
    &= \E{\forward}{\estimator{\{\suptime{\theta_i}{t}, \suptime{x_i}{t}\}_{i \in [\popSize]}}}.
\end{align}
By the assumption, the right hand side is equal to the gradient of the smoothed objective function.

We next prove~\eqref{eq:generalized-variance-reduction-of-rao-blackwellized-estimator}.
For notational simplicity, we omit the subscript $i \in [\popSize]$ and the superscript $t$ indicating the generation in the following equations.
Let us define
\begin{align}
    u := \E{\forward}{\estimator{\{\theta_i, x_i\}}} = \E{\forward}{\RaoBlackwellize{\estimatorChar}{\{x_i\}}}.
\end{align}
By direct calculation, we have
\begin{align}
    &\V{\forward}{\estimator{\{\theta_i, x_i\}}}\\
    &= \E{\forward}{\left(\estimator{\{\theta_i, x_i\}} - u\right)^2} \\
    &= \E{\forward}{\left[\left(\estimator{\{\theta_i, x_i\}} - \RaoBlackwellize{\estimatorChar}{\{x_i\}}\right) + \left(\RaoBlackwellize{\estimatorChar}{\{x_i\}} - u\right)\right]^2} \\
    &= \E{\forward}{\left(\estimator{\{\theta_i, x_i\}} - \RaoBlackwellize{\estimatorChar}{\{x_i\}}\right)^2} \\
        &\quad + 2 \E{\forward}{\left(\estimator{\{\theta_i, x_i\}} - \RaoBlackwellize{\estimatorChar}{\{x_i\}}\right) \left(\RaoBlackwellize{\estimatorChar}{\{x_i\}} - u\right)}\\
        &\quad + \E{\forward}{\left(\RaoBlackwellize{\estimatorChar}{\{x_i\}} - u\right)^2} \\
        &= \E{\forward}{\left(\estimator{\{\theta_i, x_i\}} - \RaoBlackwellize{\estimatorChar}{\{x_i\}}\right)^2} + \V{\forward}{\RaoBlackwellize{\estimatorChar}{\{x_i\}}} \\
            &\quad + 2 \E{\forward}{\left(\estimator{\{\theta_i, x_i\}} - \RaoBlackwellize{\estimatorChar}{\{x_i\}}\right) \left(\RaoBlackwellize{\estimatorChar}{\{x_i\}} - u\right)}.
\end{align}
Since the first term is non-negative, it suffices to show that the last term is zero.
This is shown by the following calculation:
\begin{align}
    &\E{\forward}{\left(\estimator{\{\theta_i, x_i\}} - \RaoBlackwellize{\estimatorChar}{\{x_i\}}\right) \left(\RaoBlackwellize{\estimatorChar}{\{x_i\}} - u\right)}\\
    &= \E{\forward[\{x_i\}]}{
            \E{\forward[\{ \theta_i \} \mid \{x_i\}]}{
                \left(\estimator{\{\theta_i, x_i\}} - \RaoBlackwellize{\estimatorChar}{\{x_i\}}\right) \left(\RaoBlackwellize{\estimatorChar}{\{x_i\}} - u\right)
            }
        } \\
    &= \E{\forward[\{\suptime{x_i}{t}\}]}{
            \E{\forward[\{\suptime{\theta_i}{t}\} \mid \{\suptime{x_i}{t}\}]}{
                \left(\estimator{\{\suptime{\theta_i}{t}, \suptime{x_i}{t}\}} - \RaoBlackwellize{\estimatorChar}{\{\suptime{x_i}{t}\}}\right)
            }
            \left(\RaoBlackwellize{\estimatorChar}{\{\suptime{x_i}{t}\}} - u\right)
        } \\
    &= \E{\forward[\{\suptime{x_i}{t}\}]}{
            \left(\RaoBlackwellize{\estimatorChar}{\{\suptime{x_i}{t}\}} - \RaoBlackwellize{\estimatorChar}{\{\suptime{x_i}{t}\}}\right)
            \left(\RaoBlackwellize{\estimatorChar}{\{\suptime{x_i}{t}\}} - u\right)
        } \\
    &= \E{\forward[\{\suptime{x_i}{t}\}]}{
            0 \cdot
            \left(\RaoBlackwellize{\estimatorChar}{\{\suptime{x_i}{t}\}} - u\right)
        } \\
    &= 0.
\end{align}
\end{proof}

For the proof of Theorem~\ref{thm:rao-blackwellized-genotype-gradient-estimator}, we need the following lemma.
\begin{lemma}[Fisher's identity~\cite{cappe2005inference}]
    \label{lem:gen-full-log-derivative}
    \begin{align}
        \E{\forward(\theta \mid x)}{\frac{\partial \log \searchDist{\psi}(\theta)}{\partial \psi}} = \frac{\partial \log \forward(x)}{\partial \psi}.
    \end{align}
\end{lemma}
\begin{proof}
\begin{align}
    \E{\forward(\theta \mid x)}{\frac{\partial \log \searchDist{\psi}(\theta)}{\partial \psi}}
    &= \int
            \frac{\frac{\partial  \searchDist{\psi}(\theta)}{\partial \psi}}{\searchDist{\psi}(\theta)}
            \forward(\theta \mid x)
        \dd \theta \\
    &= \int
    \frac{\frac{\partial  \searchDist{\psi}(\theta)}{\partial \psi}}{\searchDist{\psi}(\theta)}
    \frac{\searchDist{\psi}(\theta) \sampleDist{\theta}(x)}{\forward(x)}
    \dd \theta \\
    &= \int
    \frac{\partial  \searchDist{\psi}(\theta)}{\partial \psi}
    \frac{ \sampleDist{\theta}(x)}{\forward(x)}\\
    &= \frac{1}{\forward(x)} \frac{\partial}{\partial \psi} \int \searchDist{\psi}(\theta) \sampleDist{\theta}(x) \dd \theta \\
    &= \frac{1}{\forward(x)} \frac{\partial}{\partial \psi} \forward(x) \\
    &= \frac{\partial \log \forward(x)}{\partial \psi}.
\end{align}
\end{proof}

\begin{proof}[Proof of Theorem~\ref{thm:rao-blackwellized-genotype-gradient-estimator}]
By Lemma~\ref{lem:gen-full-log-derivative}, we have
\begin{align}
    &\RaoBlackwellizeNoArgs{\genotypeEstimator \smoothedObjFunc_{W}}\notag\\
    &= \E{\forward(\theta \mid x)}{\genotypeEstimator \smoothedObjFunc_{W}}\notag\\
    &=  \E{\forward(\theta_i \mid x_i)}{\frac{1}{\popSize} \sum_{i=1}^{\popSize} W(f(x_i)) \frac{\partial \log \searchDist{\psi}(\theta_i)}{\partial \psi}}  \notag\\
    &= \frac{1}{\popSize} \sum_{i=1}^{\popSize} W(f(x_i)) \E{\forward(\theta_i \mid x_i)}{ \frac{\partial \log \searchDist{\psi}(\theta_i)}{\partial \psi}}  \notag\\
    &= \frac{1}{\popSize} \sum_{i=1}^{\popSize} W(f(x_i)) \frac{\partial \log \forward(x_i)}{\partial \psi}\\
    &= \fullGradientEstimator \smoothedObjFunc_{W}.
    \end{align}
\end{proof}

\subsection{Proofs in Section~\ref{sec:analytical-form-of-full-gradient-estimator}}
\label{subsec:analytical-proof}

We prove~\eqref{eq:analytical-form-of-full-gradient-estimator-for-exponential-family}.
  By direct calculation, we have
\begin{align}
    \forward(x) &=  
        \baseMeasureExpFamilyChar(x) \int \exp\left[\left( \sufficientStatistic{x} + \nu \phi \right)\cdot \theta - (\nu + 1) A(\theta) - B(\phi, \nu)\right]
        \dd \theta\\
    &= \baseMeasureExpFamilyChar(x) \int \exp\left[\nu' \phi' \cdot \theta - \nu' A(\theta) -  B(\phi, \nu)\right] \dd \theta.
\end{align}
By using the definition of the normalization constant $B(\phi, \nu)$, we have
\begin{align}
    B(\phi', \nu') = \log \int \exp\left[\nu' \phi' \cdot \theta - \nu' A(\theta)\right] \dd \theta.
\end{align}
Therefore, we have
\begin{align}
    \forward(x) &=  \baseMeasureExpFamilyChar(x) \exp\left[B(\phi', \nu')  - B(\phi, \nu)\right],
\end{align}
By using this result, we have
\begin{align}
    \frac{\partial \log \forward(x)}{\partial \psi} 
    &= \frac{\partial}{\partial \psi} \left[ B(\phi', \nu') - B(\phi, \nu)\right].
\end{align}
By using the following property of the normalization constant $B(\phi, \nu)$~\cite{amari2016information},
\begin{align}
    \frac{\partial B(\phi, \nu)}{\partial \phi} = \E{\searchDist{\psi}(\theta)}{\nu \theta},\\
    \frac{\partial B(\phi, \nu)}{\partial \nu} = \E{\searchDist{\psi}(\theta)}{\phi \cdot \theta - A(\theta)},
\end{align}
we have~\eqref{eq:analytical-form-of-full-gradient-estimator-for-exponential-family}.

\subsection{Proofs in Section~\ref{sec:example-of-paes}}
\label{subsec:example-of-paes-proof}
\subsubsection{Proof of Subsection~\ref{subsec:gaussian-search-and-sampling}}
\label{subsec:gaussian-search-and-sampling-proof}
We prove Equation~\eqref{eq:natural-full-gradient-estimator-for-gaussian-search-and-sampling}.
As shown in~\cite{ollivier2017information, akimoto2010bidirectional, glasmachers2010exponential}, we have
\begin{align}
    &\FisherInfo{\forward(x)}^{-1} \frac{\partial \log \forward(x)}{\partial \bar \theta} =  (x - \bar \theta),\\
    &\FisherInfo{\forward(x)}^{-1} \frac{\partial \log \forward(x)}{\partial ( \Sigma + \sigma^2 \identityMatrix )} = (x - \bar \theta)(x - \bar \theta)^{\transpose} - \Sigma - \sigma^2 \identityMatrix.
\end{align}
They also showed that the Fisher information matrix $\FisherInfo{\forward(x)}$ is block-diagonal for $\bar \theta$ and $\Sigma + \sigma^2 \identityMatrix$.
Since $\frac{\partial (\Sigma + \sigma^2 \identityMatrix)}{\partial \Sigma}$ is identity, the chain rule implies that the Fisher information matrix and the above partial derivatives are invariant when we change the coordinate system from $\Sigma + \sigma^2 \identityMatrix$ to $\Sigma$.
Thus, we have~\eqref{eq:natural-full-gradient-estimator-for-gaussian-search-and-sampling}.

\subsubsection{Proof of Subsection~\ref{subsec:example-classification-with-bandit-feedback}}
\label{subsec:example-classification-with-bandit-feedback-proof}
We prove Equation~\eqref{eq:full-gradient-estimator-for-classification-with-bandit-feedback}.
For notational simplicity, we define 
\begin{align}
    \mathcal{B}(c, \alpha) = \mathcal{B}(\{\alpha_{x,c}\}_{x \in \SampleSp}).
\end{align}

By direct calculation, we have
\begin{align}
    \forward(x \mid c) &= \int \searchDist{\psi}(\theta \mid c) \sampleDist{\theta}(x \mid c) \dd \theta\\
    &= \int \frac{1}{\mathcal{B}(c, \alpha)} \left( \prod_{x' \in \SampleSp} \pi(x' \mid c)^{\alpha_{x',c} - 1} \right) \pi(x \mid c) \dd \pi(\cdot \mid c)\\
    &= \frac{1}{\mathcal{B}(c,\alpha)} \int \prod_{x' \in \SampleSp} \pi(x' \mid c)^{\alpha_{x',c} + \delta_{x, x'} - 1}  \dd \pi(\cdot \mid c).
\end{align}
By using the normalization condition
\begin{align}
    \mathcal{B}(c, \alpha) = \int \prod_{x' \in \SampleSp} \pi(x' \mid c)^{\alpha_{x',c} - 1} \dd \pi(\cdot \mid c),
\end{align}
we can calculate the integration analytically and we have
\begin{align}
    \label{eq:forward-for-classification-with-bandit-feedback}
    \forward(x \mid c) = \frac{\mathcal{B}(c, \alpha')}{\mathcal{B}(c, \alpha)},
\end{align}
where $\alpha' = \{\alpha_{x',c} + \delta_{x, x'}\}_{x' \in \SampleSp}$.
By inserting this result into~\eqref{eq:full-gradient-estimator}, we have
\begin{align}
    &\fullGradientGeneralArg{\alpha_{x,c}} \smoothedObjFunc_{W} \\
    &= \frac{1}{\popSize}\sum_{i=1}^{\popSize}  W(f(x_i)) \left( \diGamma(\alpha_{x,c} + \delta_{c, c_i} \delta_{x, x_i}) - \diGamma(\alpha_{x,c}) + \diGamma(\sum_{x' \in \SampleSp} \alpha_{x',c}) - \diGamma(\sum_{x' \in \SampleSp} \alpha_{x',c} + \delta_{c, c_i} \delta_{x, x_i}) \right),
\end{align}
where $\diGamma$ is the digamma function defined by
\begin{align}
    \diGamma(x) = \frac{\Gamma'(x)}{\Gamma(x)}.
\end{align}
A property of the digamma function is that $\diGamma(x+1) = \diGamma(x) + \frac{1}{x}$. 
By using this property, we have~\eqref{eq:full-gradient-estimator-for-classification-with-bandit-feedback}.

\subsubsection{Proof of Subsection~\ref{subsec:example-reinforcement-learning}}
\label{subsec:example-reinforcement-learning-proof}
We prove Equation~\eqref{eq:full-gradient-estimator-for-reinforcement-learning-with-discrete-state-and-action}.
By a similar argument to~\eqref{eq:forward-for-classification-with-bandit-feedback}, we have
\begin{align}
    \forward(x) = \left(\prod_{t=0}^{T-1} \transitMatrix(s_{t+1} \mid s_t, a_t)\right) \prod_{s \in \stateSp} \frac{\mathcal{B}\left((\alpha_{a,s} + j(a,s))_{a \in \actionSp}\right)}{\mathcal{B}\left((\alpha_{a,s})_{a \in \actionSp}\right)}.
\end{align}
From this equation, we have
\begin{align}
    &\fullGradientGeneralArg{\alpha_{a,s}} \smoothedObjFunc_{W} \\
    &= \frac{1}{\popSize}\sum_{i=1}^{\popSize} W(f(x_i)) \left( \diGamma(\alpha_{a,s} + j_i(a,s)) - \diGamma(\alpha_{a,s}) + \diGamma(\sum_{a' \in \actionSp} \alpha_{a',s}) - \diGamma\left(\sum_{a' \in \actionSp} \left(\alpha_{a',s} + j_i(a',s)\right)\right) \right),
\end{align}

By using the property of the digamma function $\diGamma(x+1) = \diGamma(x) + \frac{1}{x}$ iteratively, we have~\eqref{eq:full-gradient-estimator-for-reinforcement-learning-with-discrete-state-and-action}.

\subsubsection{Proof of Subsection~\ref{subsec:example-rl-reparameterization}}
\label{subsubsec:example-rl-reparameterization-proof}
We prove Equation~\eqref{eq:reparameterized-gradient-estimator-example}.

We have
\begin{align}
    &\left[ \frac{\partial \log \sampleDist{\theta}(x)}{\partial \theta_{a,s}} \right]_{\theta = \thetaFunc{\psi}{\noise}}\\
    &= \sum_{a \in \actionSp, s \in \stateSp} \left[ \frac{\partial}{\partial \theta_{a,s}} j(a, s) \log \pi_{\theta}(a \mid s) \right]_{\theta = \thetaFunc{\psi}{\noise}}\\
    &= \sum_{a \in \actionSp, s \in \stateSp} \frac{\partial}{\partial \theta_{a,s}} j(a, s) \left[ \theta_{a,s} - \log \sum_{a' \in \actionSp} \exp(\theta_{a',s})\right]_{\theta = \thetaFunc{\psi}{\noise}}\\
    &=  j(a, s) - j(s) \pi_{\thetaFunc{\psi}{\noise}}(a \mid s).
\end{align}

Also,
\begin{align}
    \frac{\partial \thetaFunc{\psi}{\noise}}{\partial \psi} = \frac{\partial}{\partial \bar \theta}\left[\bar \theta + \sigma \noise\right] = \identityMatrix.
\end{align}

By inserting these results into~\eqref{eq:reparameterized-gradient-estimator}, we have~\eqref{eq:reparameterized-gradient-estimator-example}.

\section{Additional Results on OIU with Gaussian Noise}
\label{sec:appendix-bbob}

We provide the convergence curves for all 24 BBOB functions under Gaussian OIU setting (Section~\ref{sec:numerical-experiments}) (Figure~\ref{fig:bbob-appendix-1} and Figure~\ref{fig:bbob-appendix-2}).
Each curve is the mean of 10 different random seeds, and the shaded area shows the standard deviation (std).
We also provide the computation time for 5000 generations for ES and PAES on all 24 BBOB functions (Figure~\ref{fig:computation-time}).

\begin{figure}[htbp]
  \centering
  \begin{minipage}[t]{0.32\textwidth}
    \includegraphics[width=\textwidth]{fig/experiment/bbob/bbob_f01_aggregated}
  \end{minipage}\hfill
  \begin{minipage}[t]{0.32\textwidth}
    \includegraphics[width=\textwidth]{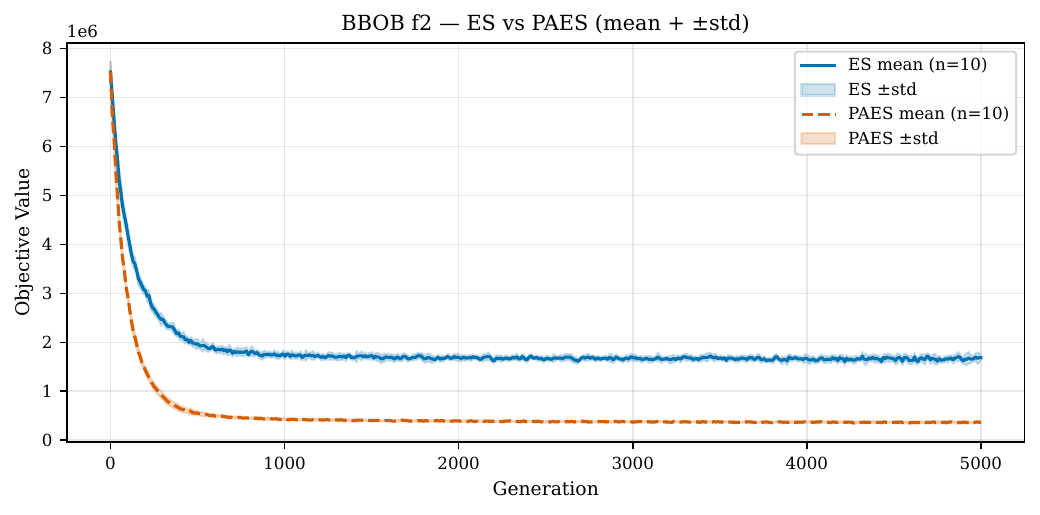}
  \end{minipage}\hfill
  \begin{minipage}[t]{0.32\textwidth}
    \includegraphics[width=\textwidth]{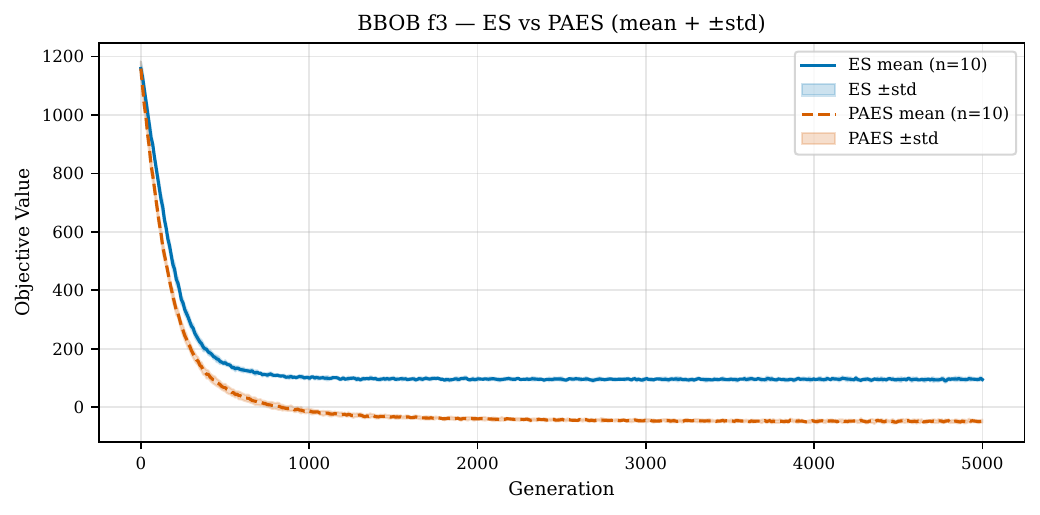}
  \end{minipage}

  \vspace{0.5em}

  \begin{minipage}[t]{0.32\textwidth}
    \includegraphics[width=\textwidth]{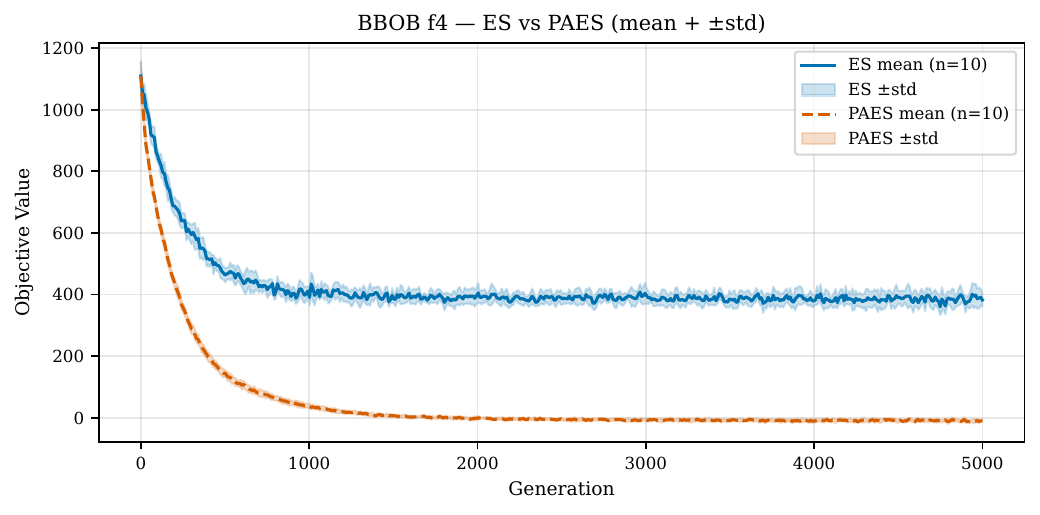}
  \end{minipage}\hfill
  \begin{minipage}[t]{0.32\textwidth}
    \includegraphics[width=\textwidth]{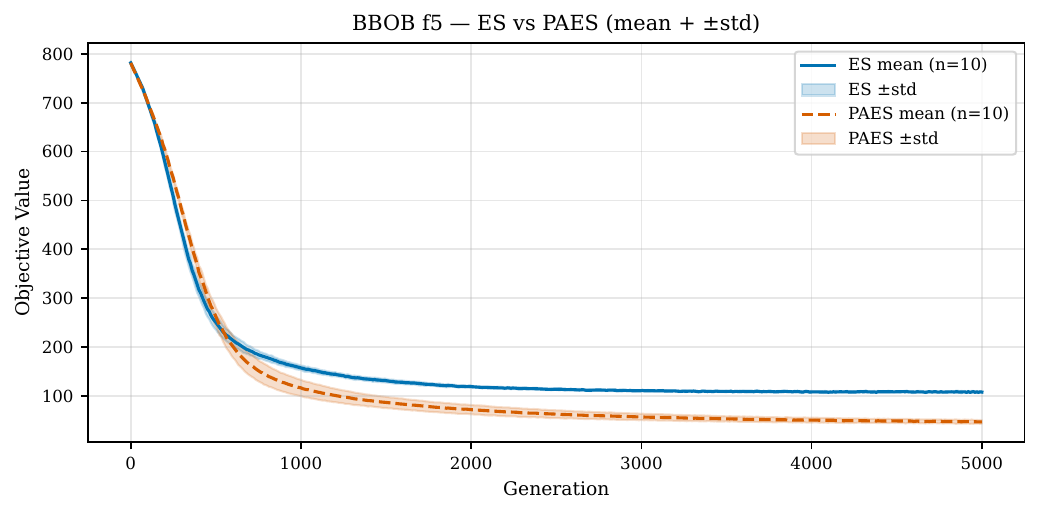}
  \end{minipage}\hfill
  \begin{minipage}[t]{0.32\textwidth}
    \includegraphics[width=\textwidth]{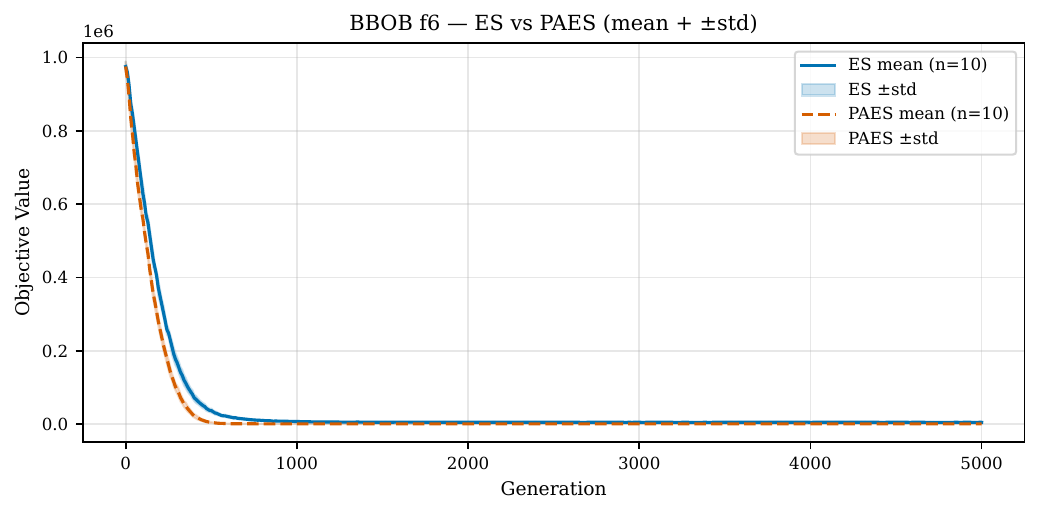}
  \end{minipage}

  \vspace{0.5em}

  \begin{minipage}[t]{0.32\textwidth}
    \includegraphics[width=\textwidth]{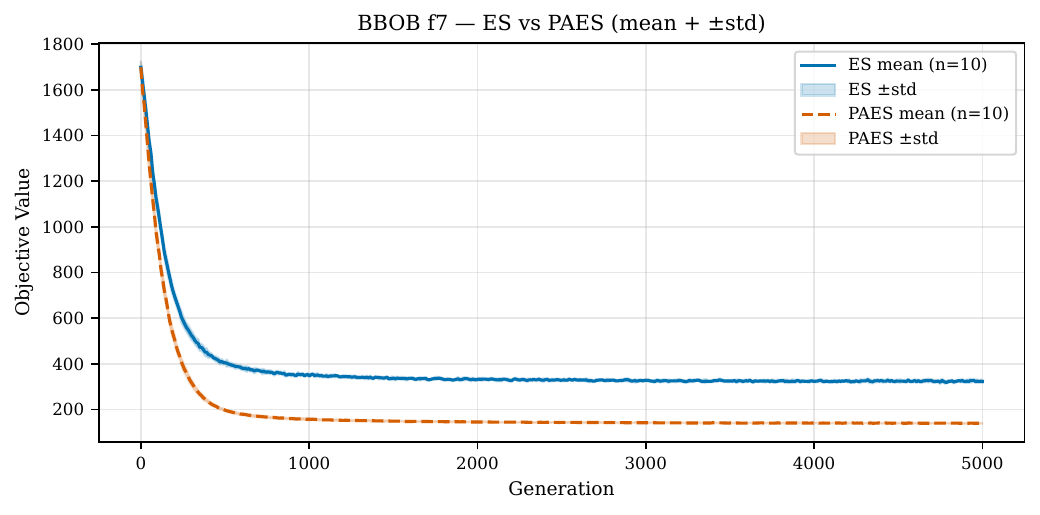}
  \end{minipage}\hfill
  \begin{minipage}[t]{0.32\textwidth}
    \includegraphics[width=\textwidth]{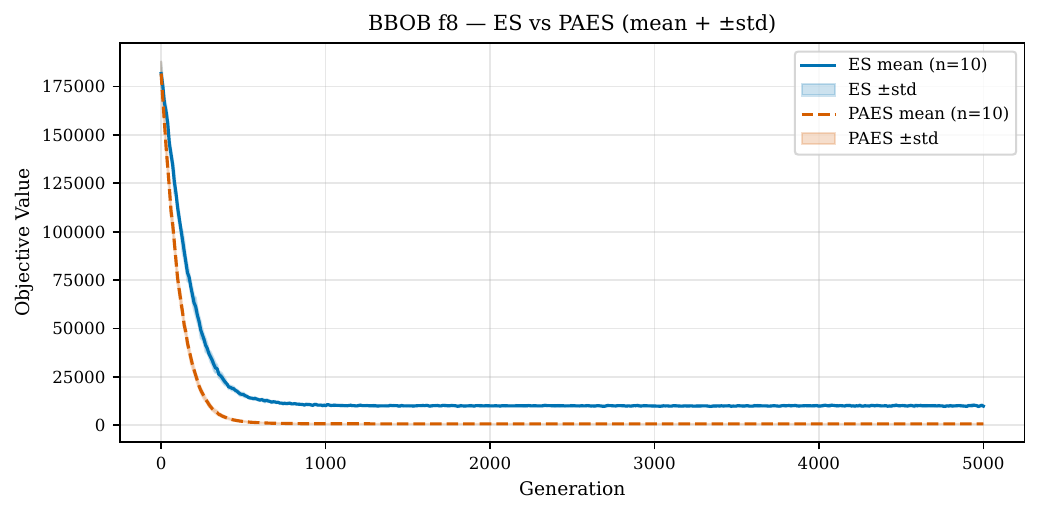}
  \end{minipage}\hfill
  \begin{minipage}[t]{0.32\textwidth}
    \includegraphics[width=\textwidth]{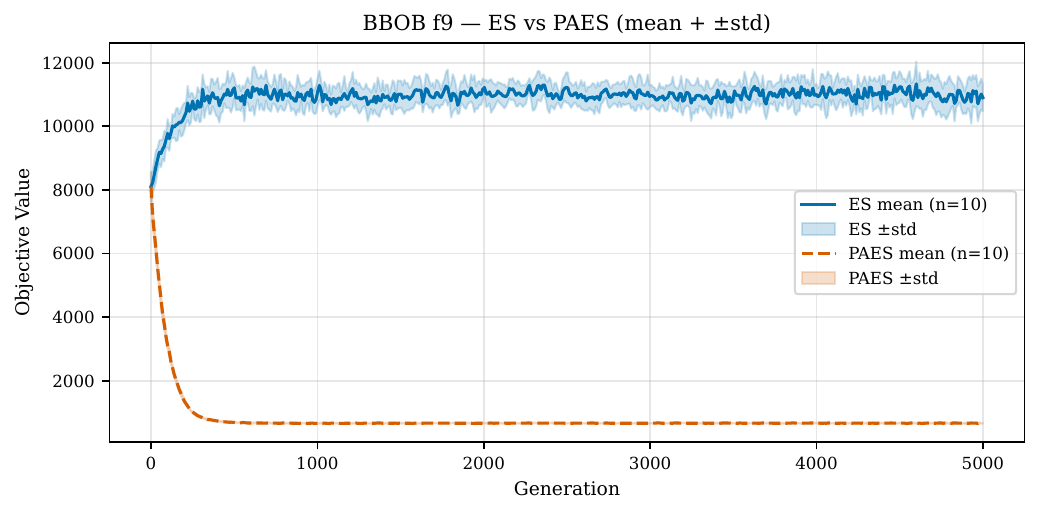}
  \end{minipage}

  \vspace{0.5em}

  \begin{minipage}[t]{0.32\textwidth}
    \includegraphics[width=\textwidth]{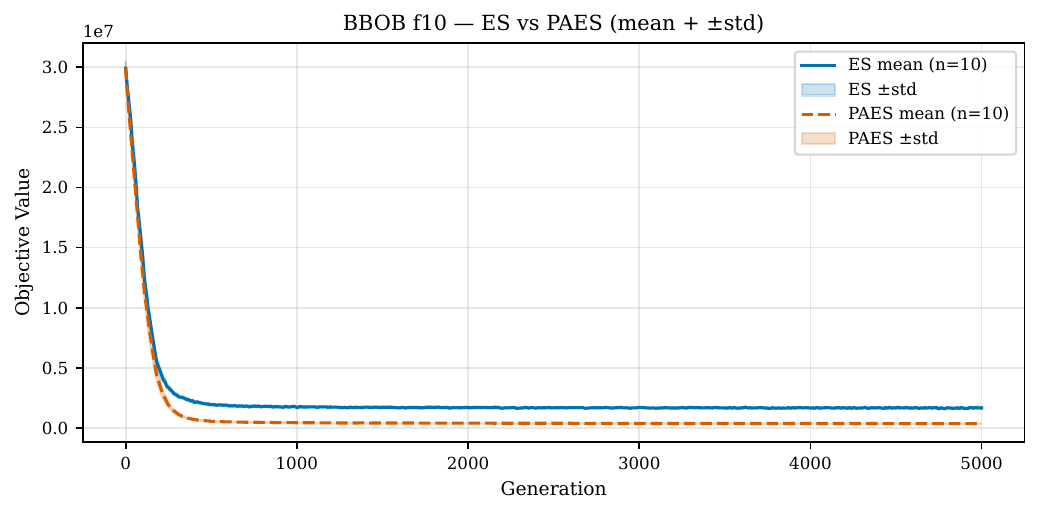}
  \end{minipage}\hfill
  \begin{minipage}[t]{0.32\textwidth}
    \includegraphics[width=\textwidth]{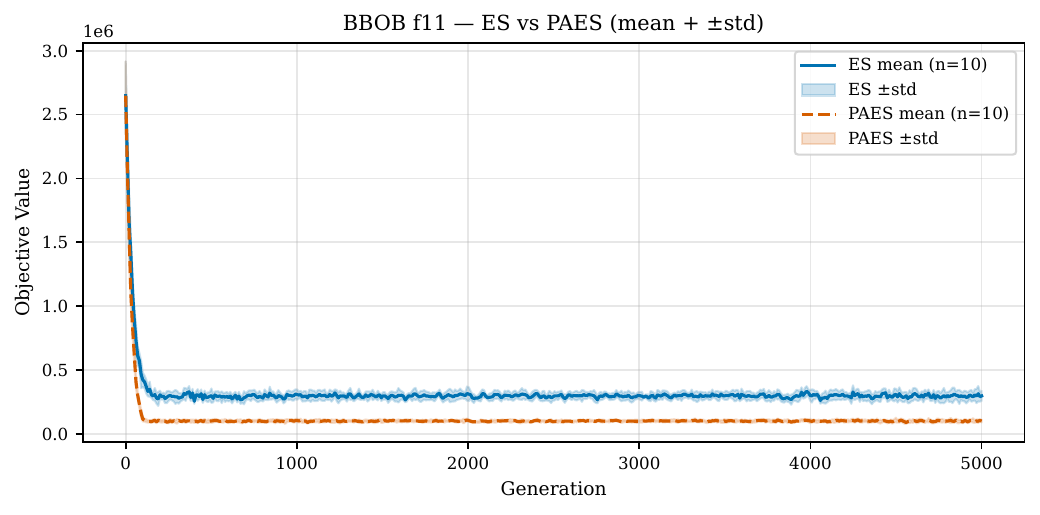}
  \end{minipage}\hfill
  \begin{minipage}[t]{0.32\textwidth}
    \includegraphics[width=\textwidth]{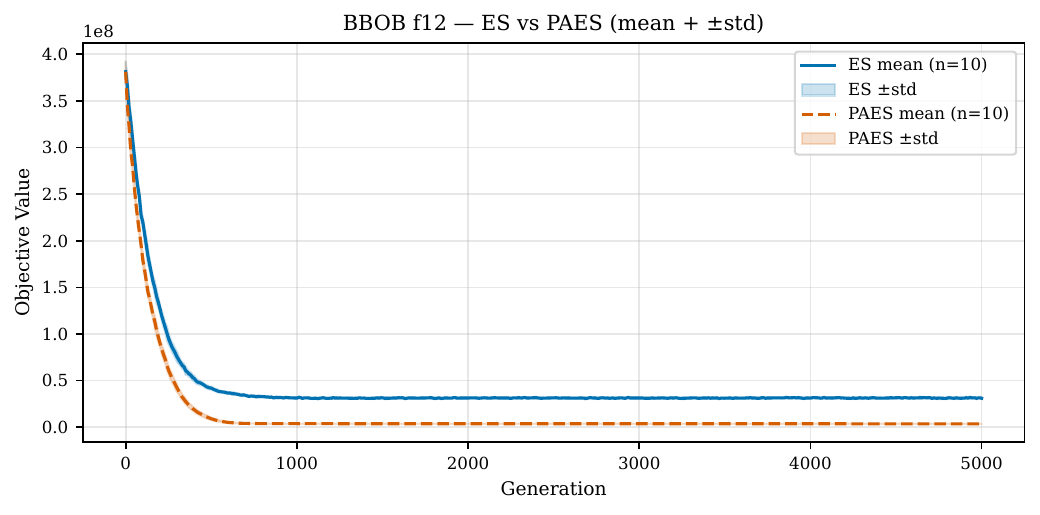}
  \end{minipage}
  \caption{BBOB f1--f12: convergence curves (ES vs.\ PAES, mean $\pm$ std, 10 seeds,
    $d=40$, population size $100$, up to $5000$ generations).}
  \label{fig:bbob-appendix-1}
\end{figure}

\begin{figure}[htbp]
  \centering
  \begin{minipage}[t]{0.32\textwidth}
    \includegraphics[width=\textwidth]{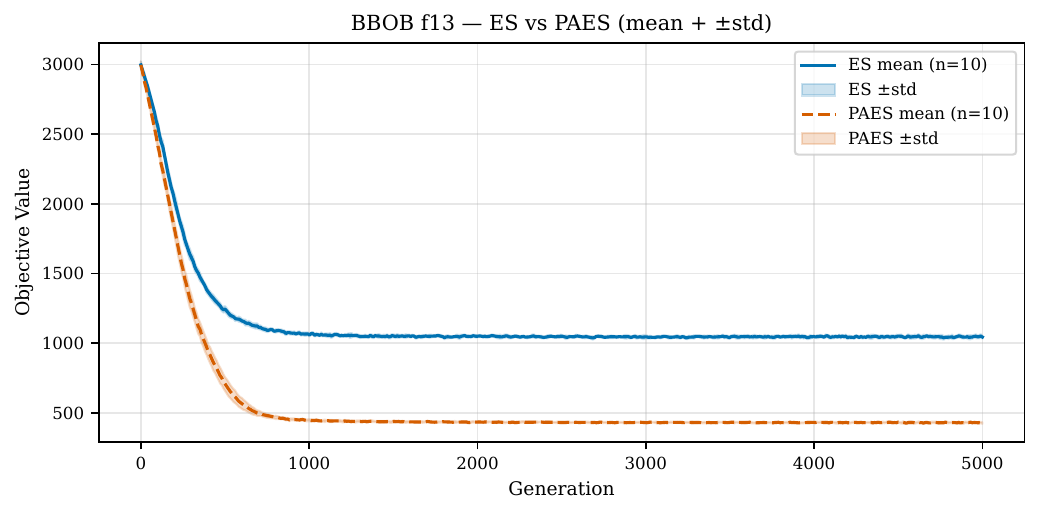}
  \end{minipage}\hfill
  \begin{minipage}[t]{0.32\textwidth}
    \includegraphics[width=\textwidth]{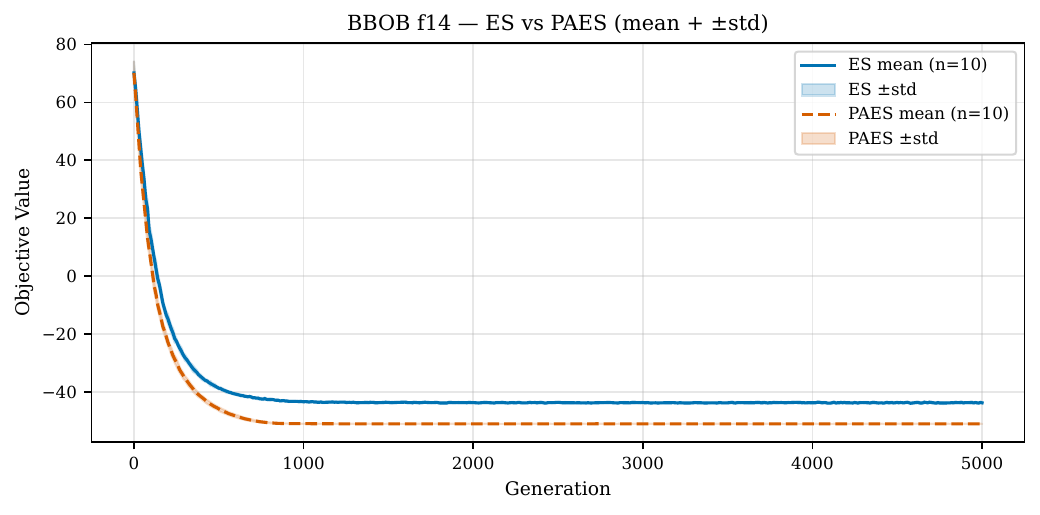}
  \end{minipage}\hfill
  \begin{minipage}[t]{0.32\textwidth}
    \includegraphics[width=\textwidth]{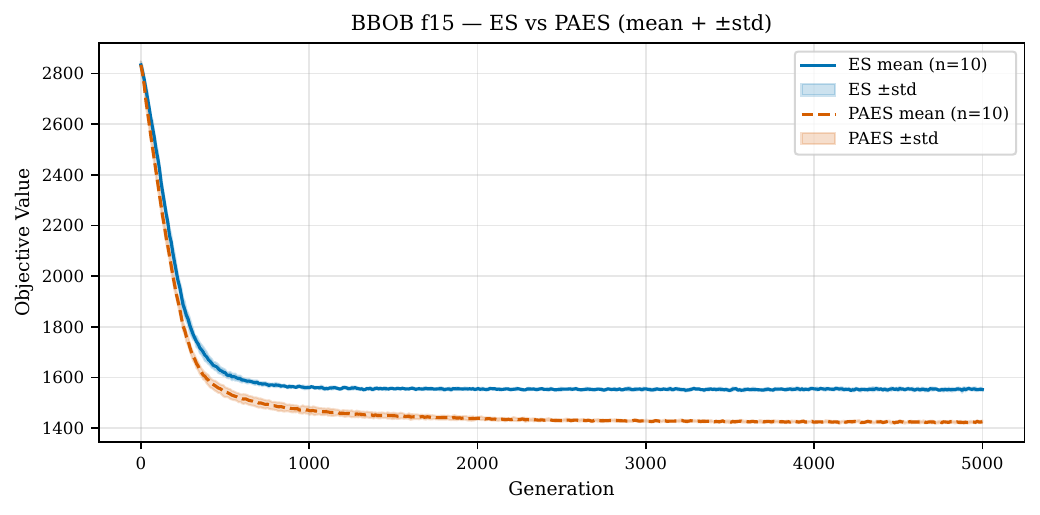}
  \end{minipage}

  \vspace{0.5em}

  \begin{minipage}[t]{0.32\textwidth}
    \includegraphics[width=\textwidth]{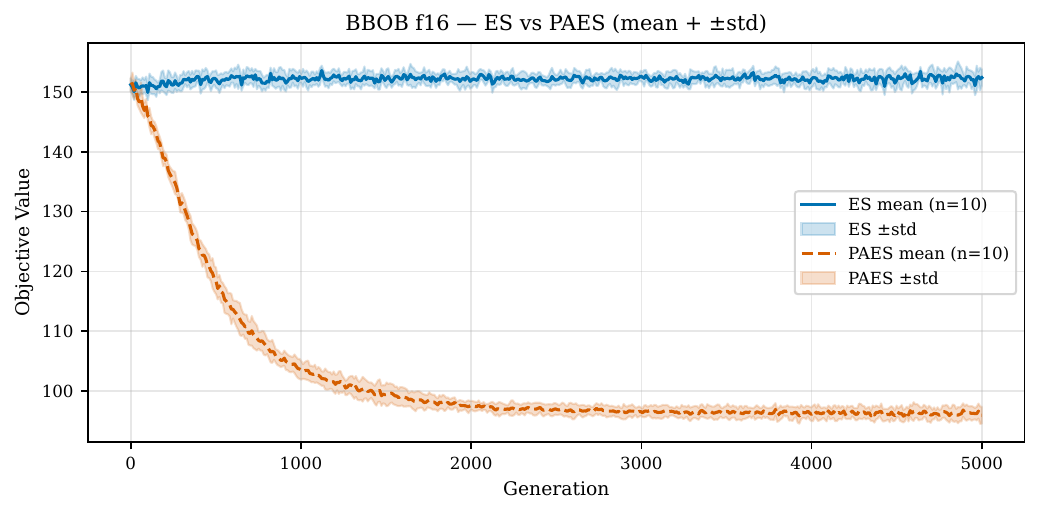}
  \end{minipage}\hfill
  \begin{minipage}[t]{0.32\textwidth}
    \includegraphics[width=\textwidth]{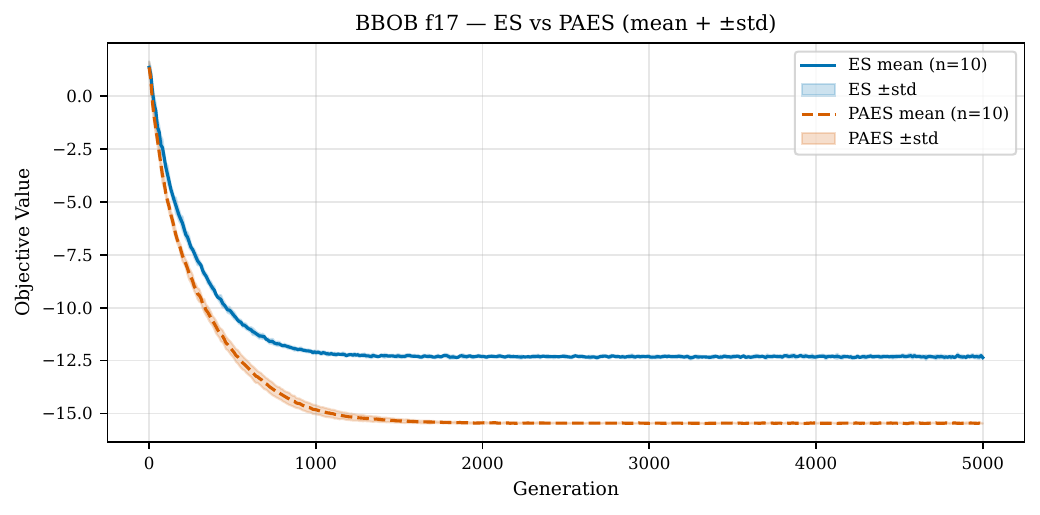}
  \end{minipage}\hfill
  \begin{minipage}[t]{0.32\textwidth}
    \includegraphics[width=\textwidth]{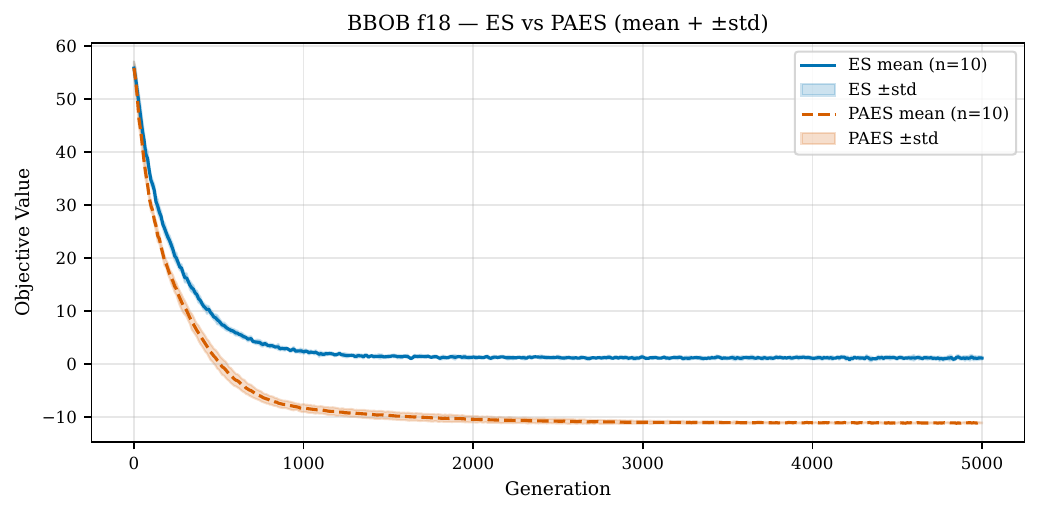}
  \end{minipage}

  \vspace{0.5em}

  \begin{minipage}[t]{0.32\textwidth}
    \includegraphics[width=\textwidth]{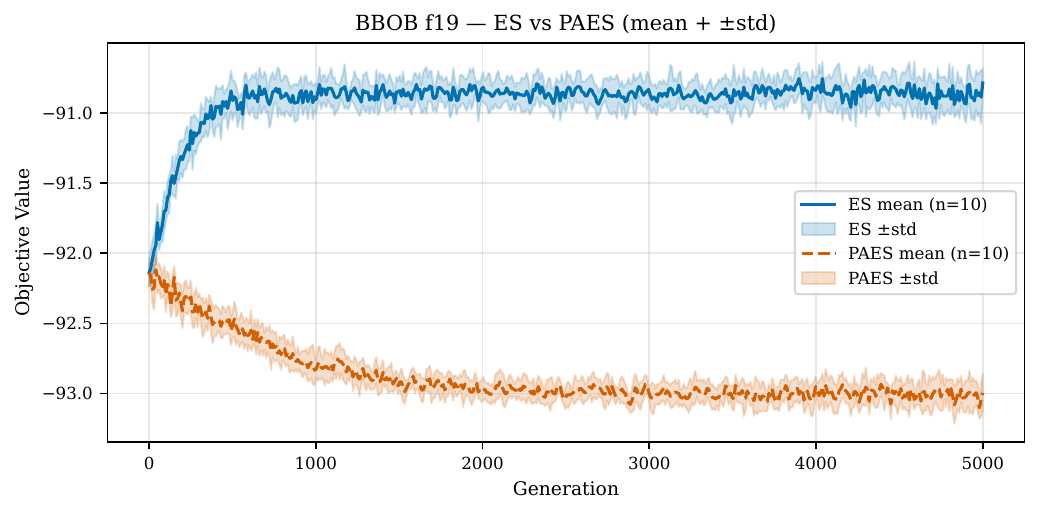}
  \end{minipage}\hfill
  \begin{minipage}[t]{0.32\textwidth}
    \includegraphics[width=\textwidth]{fig/experiment/bbob/bbob_f20_aggregated}
  \end{minipage}\hfill
  \begin{minipage}[t]{0.32\textwidth}
    \includegraphics[width=\textwidth]{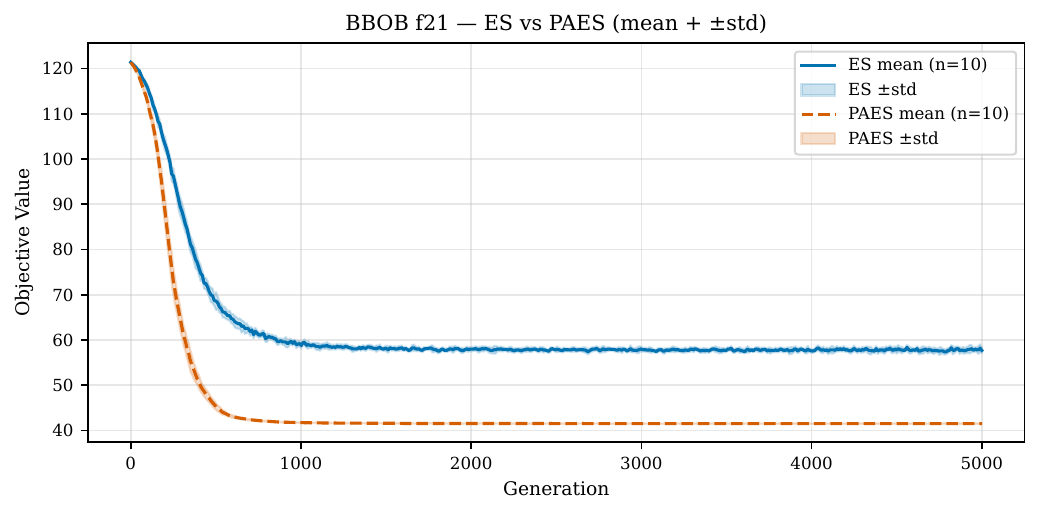}
  \end{minipage}

  \vspace{0.5em}

  \begin{minipage}[t]{0.32\textwidth}
    \includegraphics[width=\textwidth]{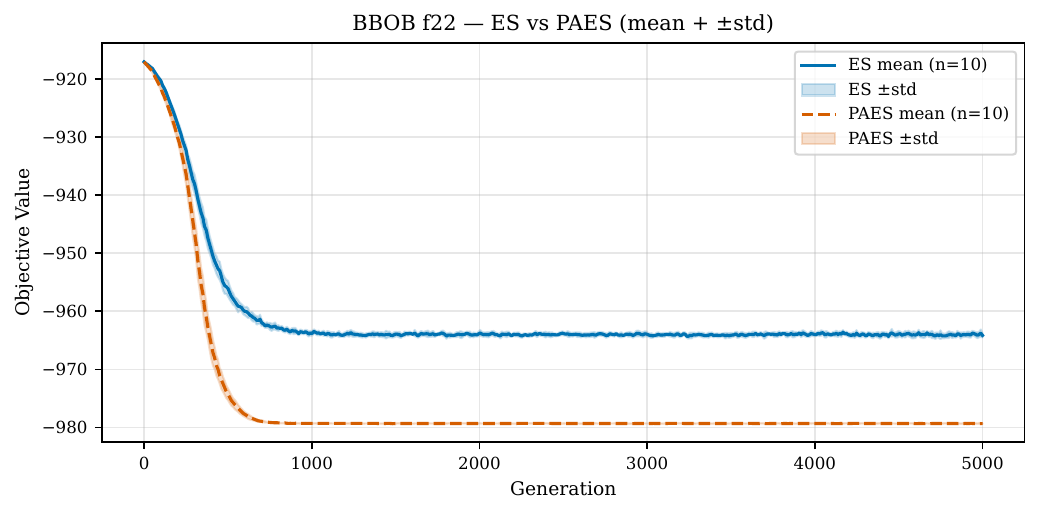}
  \end{minipage}\hfill
  \begin{minipage}[t]{0.32\textwidth}
    \includegraphics[width=\textwidth]{fig/experiment/bbob/bbob_f23_aggregated}
  \end{minipage}\hfill
  \begin{minipage}[t]{0.32\textwidth}
    \includegraphics[width=\textwidth]{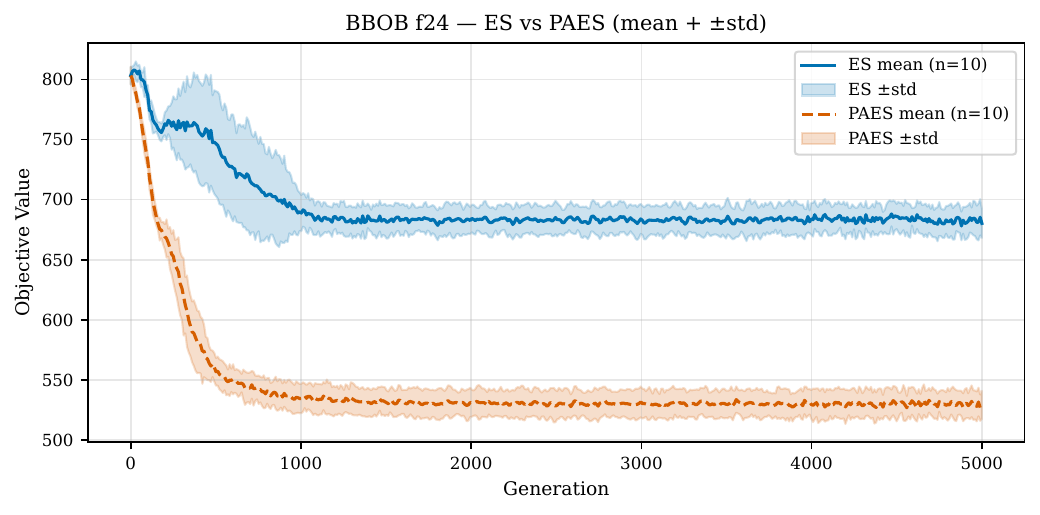}
  \end{minipage}
  \caption{BBOB f13--f24: convergence curves (ES vs.\ PAES, mean $\pm$ std, 10 seeds, dimension
    $d=40$, population size $100$, up to $5000$ generations).}
  \label{fig:bbob-appendix-2}
\end{figure}

\begin{figure}[htbp]
  \centering
  \includegraphics[width=\textwidth]{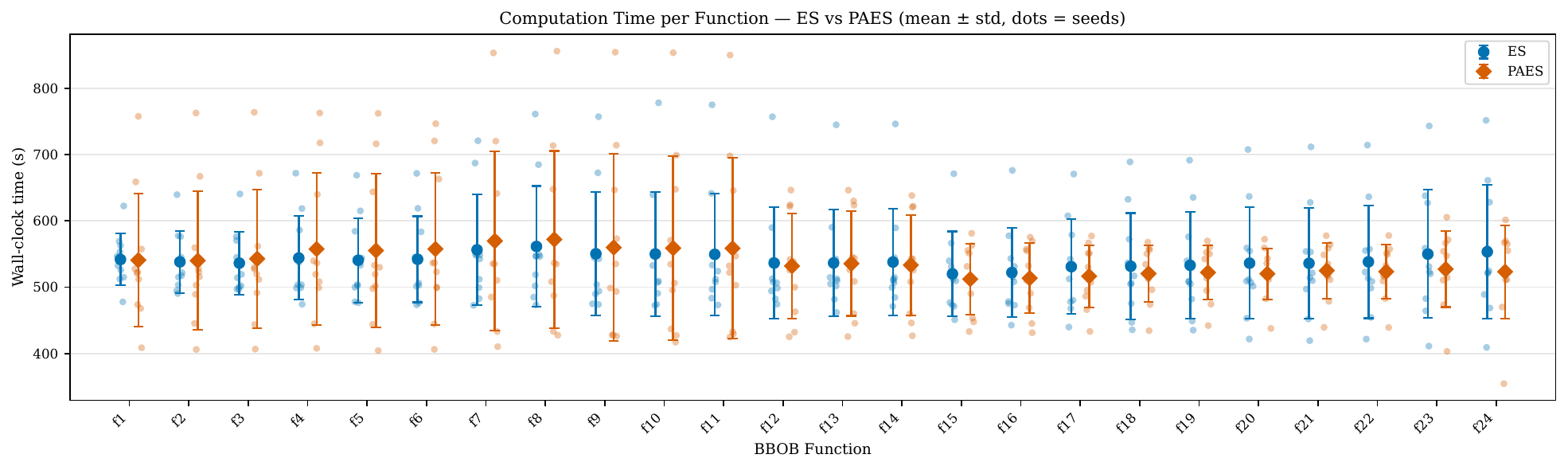}
  \caption{Computation time for 5000 generations for ES and PAES on all 24 BBOB functions (dimension $d=40$, population size $100$). Both methods have similar computation time. We show the mean and standard deviation over 10 different random seeds. Each dot represents the computation time for one run.}
  \label{fig:computation-time}
\end{figure}

\section{Additional Details of Reinforcement Learning Experiments}
\label{sec:appendix-rl-main}

\subsection{Hyperparameter Sensitivity for FrozenLake-v1}
\label{sec:appendix-rl-frozenlake-sweep}
We provide the sensitivity of learning curve of FrozenLake-v1 to the hyperparameters $\sigma$ for ES and $S = \sum_{a \in \actionSp} \alpha_{a,s}$ for PAES.
We observe that the speed of convergence of ES is faster when $\sigma$ is no less than $0.1$ (Figure~\ref{fig:frozenlake-sweep} (a)).
The speed of convergence of PAES is faster when $S$ is smaller than $1$, but the final reward decreases when $S$ is too small (Figure~\ref{fig:frozenlake-sweep} (b)).
Also, the learning curve fluculates after reaching high cumulative reward ($\approx 0.8$) when $S$ is too small ($S < 1$) (Figure~\ref{fig:frozenlake-sweep} (b)).

We note that the Fisher information matrix for the Dirichlet distribution depends on $S$ even though the ratio of $\alpha_{a,s}$ is fixed.
It means that we need to change the learning rate depending on $S$ or use natural gradient to compare the speed of convergence for different $S$ more accurately.
However, changing the learning rate and using natural gradient obscure the effect of $S$ on the learning curve.
In this paper, we  focus on the effect of the hyperparameter $S$ on the learning curve, and we do not change the learning rate dependending on $S$.

\begin{figure}[htbp]
  \centering
  \begin{minipage}[t]{0.48\textwidth}
    \centering
    \includegraphics[width=\textwidth]{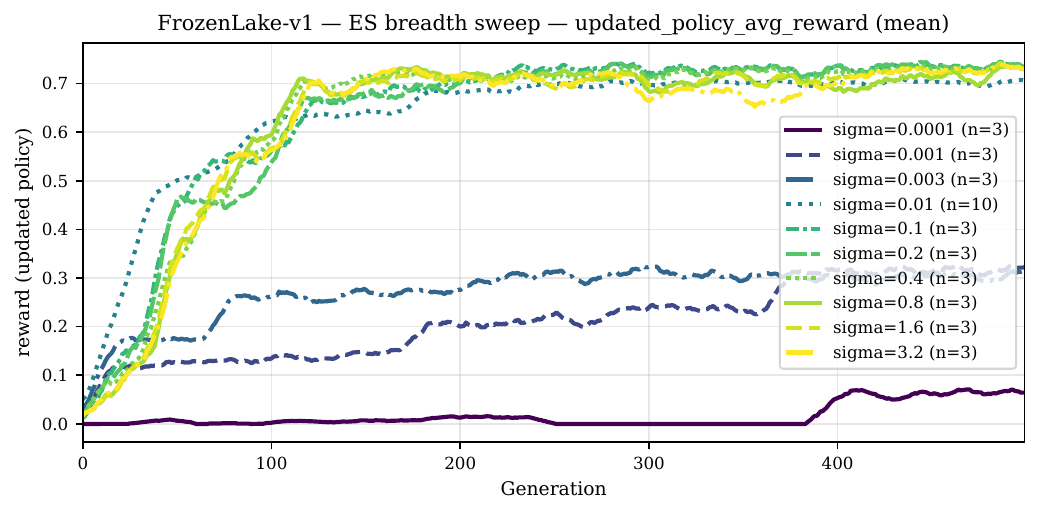}
    \centerline{\small (a) ES. }
  \end{minipage}\hfill
  \begin{minipage}[t]{0.48\textwidth}
    \centering
    \includegraphics[width=\textwidth]{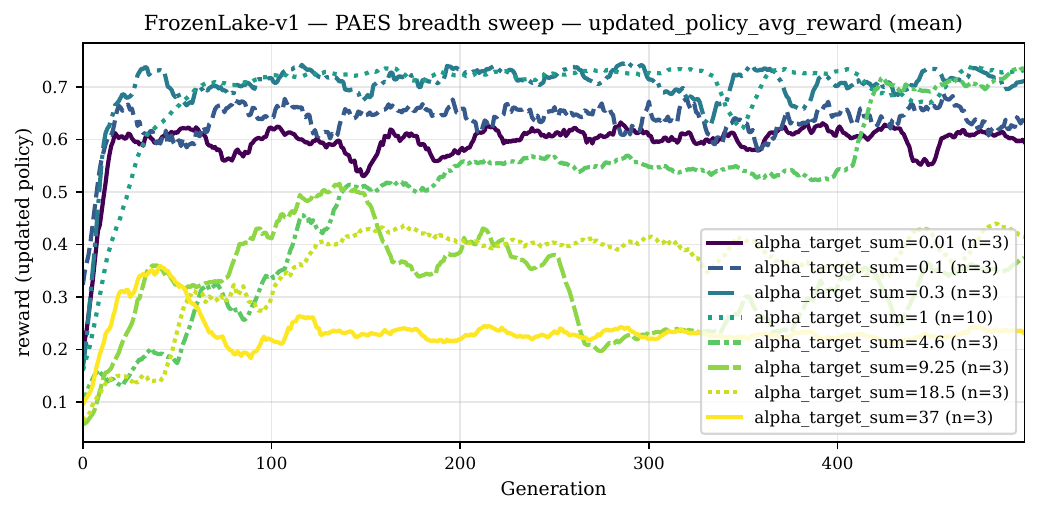}
    \centerline{\small (b) PAES.}
  \end{minipage}
  \caption{(a) The effect of the hyperparameter $\sigma$ on ES. We observe that the performance of ES is better when $\sigma$ is no less than $0.1$.   (b) The effect of the hyperparameter $S = \sum_{a \in \actionSp} \alpha_{a,s}$ on PAES. We observe that, when $S=1$, the converged cumulative reward is the highest and also the learning curve is stable after convergence.}
  \label{fig:frozenlake-sweep}
\end{figure}

\subsection{Computation Time for Reinforcement Learning Experiments}
\label{sec:appendix-rl}

We provide the computation time comparison for the RL experiments in Section~\ref{sec:rl-exact-rao-blackwell} and Section~\ref{sec:rl-reparameterization}.

\begin{figure}[htbp]
  \centering
  \begin{minipage}[t]{0.48\textwidth}
    \centering
    \includegraphics[width=\textwidth]{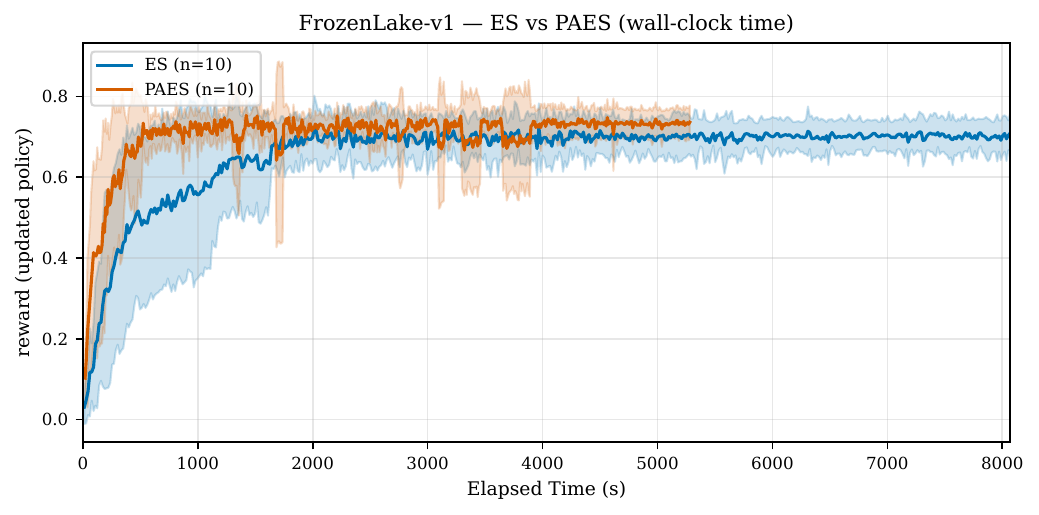}
    \centerline{\small (a) FrozenLake-v1.}
  \end{minipage}\hfill
  \begin{minipage}[t]{0.48\textwidth}
    \centering
    \includegraphics[width=\textwidth]{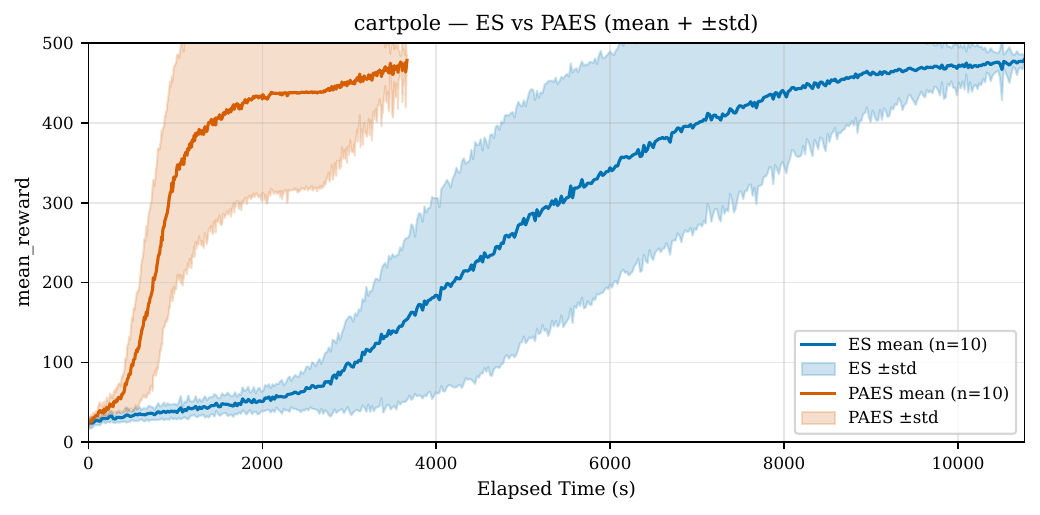}
    \centerline{\small (b) CartPole-v1.}
  \end{minipage}

  \vspace{0.8em}

  \begin{minipage}[t]{0.48\textwidth}
    \centering
    \includegraphics[width=\textwidth]{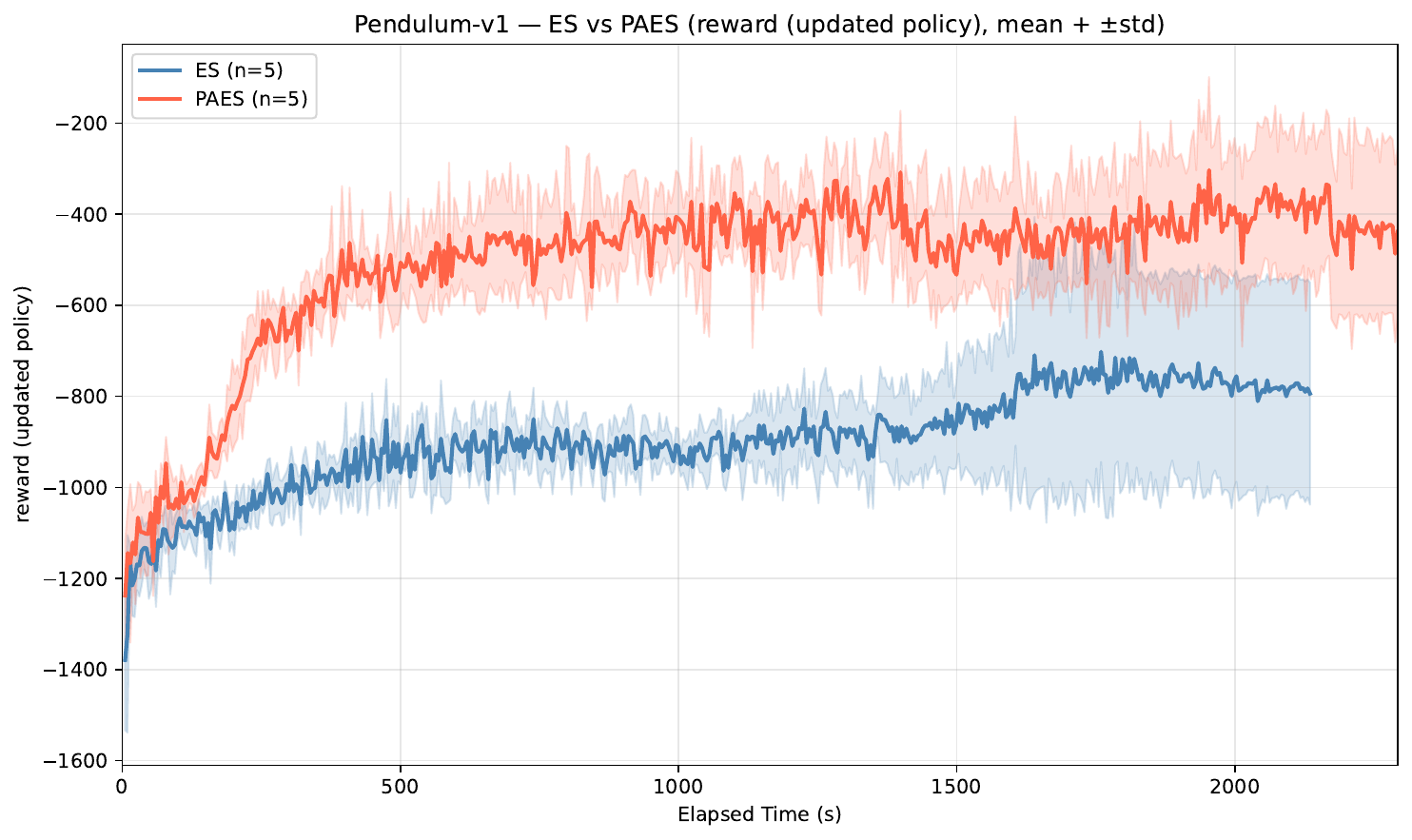}
    \centerline{\small (c) Pendulum-v1.}
  \end{minipage}\hfill
  \begin{minipage}[t]{0.48\textwidth}
    \centering
    \includegraphics[width=\textwidth]{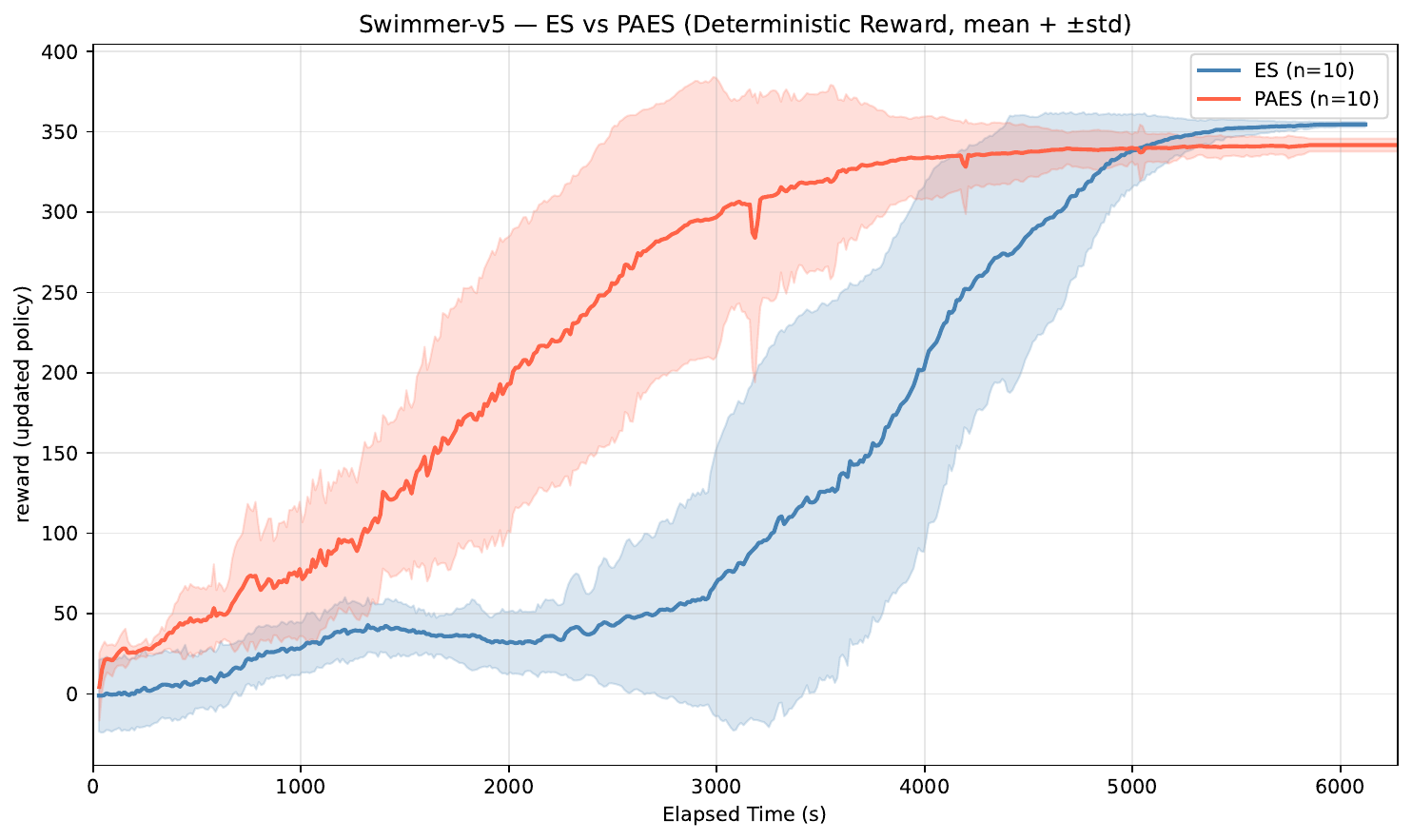}
    \centerline{\small (d) Swimmer-v5.}
  \end{minipage}
  \caption{Reward vs.\ wall-clock time for all continuous control environments.
    PAES (orange) vs.\ ES (blue), mean $\pm$ std over multiple seeds.}
  \label{fig:rl-computation-time}
\end{figure}

\end{document}